%% file: Lawson_Convergence.tex
\documentclass[11pt]{article}
\input defnAMS.tex
\usepackage{kbordermatrix}

\usepackage{caption}
\usepackage{amsfonts,amsthm}
\usepackage{amsmath,amssymb,enumerate,color}
\usepackage{algorithm,algorithmic}
\usepackage{epsfig,graphicx,psfrag,mathrsfs,subfigure}
\usepackage{eucal}
\usepackage{yhmath}
\usepackage{hyperref}
\hypersetup{colorlinks}
\usepackage{diagbox}
\RequirePackage{multirow}
 
\usepackage{textcomp}
\usepackage[]{fontenc}
\usepackage{extarrows}
\def\wtd{\widetilde}

\usepackage{rotating}
\usepackage{lscape}
\usepackage{bm}
\usepackage{xspace}
\usepackage{tikz}
\usetikzlibrary{petri} 
\usetikzlibrary{shadows,arrows,positioning,shapes.geometric} 

\usepackage{geometry}
\def\ba{\pmb{a}}
\def\bb{\pmb{b}}
\def\bc{\pmb{c}}

\def\be{\pmb{e}}
\def\Bf{\pmb{f}}
\def\bg{\pmb{g}}

\def\bp{\pmb{p}}
\def\bq{\pmb{q}}

\def\bs{\pmb{s}}
\def\bt{\pmb{t}}
\def\bu{\pmb{u}}
\def\bv{\pmb{v}}
\def\bw{\pmb{w}}
\def\bx{\pmb{x}}
\def\by{\pmb{y}}
\def\bz{\pmb{z}}

\def\diag{{\rm diag}}

\def\scrR{\mathscr{R}}

\def\wtd{\widetilde}

\def\bbC{\mathbb{C}}

\def\bbP{\mathbb{P}}
\def\bbR{\mathbb{R}}

\renewcommand{\algorithmicrequire}{\textbf{Input:}}
\renewcommand{\algorithmicensure}{\textbf{Output:}}

\numberwithin{equation}{section}
\numberwithin{figure}{section}
\numberwithin{table}{section}

\graphicspath{{figures/}{figure/}{pictures/}%
	{picture/}{pic/}{pics/}{image/}{images/}}

\title{A convergence analysis of Lawson's iteration for computing polynomial and rational minimax approximations}
\author{Lei-Hong Zhang\thanks{Corresponding author. School of Mathematical Sciences, Soochow University, Suzhou 215006, Jiangsu, China. This work was
 supported in part by the National Natural Science Foundation of China NSFC-12071332, NSFC-12371380,    Jiangsu Shuangchuang Project (JSSCTD202209), Academic Degree and Postgraduate Education Reform Project of Jiangsu Province, and China Association of Higher Education under grant 23SX0403.
        Email: {\tt longzlh@suda.edu.cn}.} \and Shanheng Han\thanks{School of Mathematical Sciences, Soochow University, Suzhou 215006, Jiangsu, China. Email: {\tt 3468805603@qq.com}.}
        }
 
 \date{ }

\begin{document}

\maketitle

\begin{abstract}
Lawson's iteration is a classical and effective method for solving the linear (polynomial) minimax approximation {problem} in the complex plane. Extension of Lawson's iteration for the rational minimax approximation {problem} with both computationally high efficiency and theoretical guarantee is challenging. A recent work  [{L.-H. Zhang, L. Yang, W. H. Yang and Y.-N. Zhang, A convex dual problem for the rational minimax approximation and Lawson's iteration,   {\it Math. Comp.},  {94(2025), 2457–2494.}}] reveals that Lawson's iteration can be viewed as a method for solving the dual problem of the original rational minimax approximation {problem}, and a new type of Lawson's iteration, namely, {\tt d-Lawson}, was proposed, which   reduces to the classical Lawson's iteration for the linear  minimax approximation {problem}. For the rational case, such a dual problem is guaranteed to obtain the original minimax solution under Ruttan's sufficient condition, and numerically, {\tt d-Lawson} was observed to converge monotonically with respect to the dual objective function. In this paper, we present a theoretical convergence analysis of {\tt d-Lawson} for both the linear and rational minimax approximation {problems}. In particular, we show that
\begin{itemize}
\item[ (i)] for the linear minimax approximation {problem}, $\beta=1$ is a near-optimal Lawson exponent  in Lawson's iteration, and 
\item[(ii)] for the rational minimax approximation {problem, under certain conditions,} {\tt d-Lawson} converges monotonically with respect to the dual objective function for any sufficiently small $\beta>0$, and the limiting approximant satisfies the complementary slackness condition: any node associated with positive weight either is an interpolation point or has a constant error. 
\end{itemize}
\end{abstract}


\medskip
{\small
\medskip

}

 
\section{Introduction}\label{sec_intro}
Computing the polynomial and/or rational minimax (also known as Chebyshev or best) approximations of a continuous complex-valued function $f$ over a given compact set $\Omega$ in the complex plane $\bbC$ is a classical problem in approximation theory \cite{tref:2019a}. In practical applications, rational approximations are useful in various areas, including function approximations \cite{drnt:2024,tref:2023},   computing conformal mappings \cite{gotr:2019a,heht:2023,tref:2020}, solving partial differential equations \cite{brtr:2022,cotr:2023,gotr:2019b,gotr:2019,natr:2021,tref:2018,tref:2024a,xuwt:2024}, model order reduction \cite{anbg:2020,dajm:2023,gogu:2021}, and signal processing  \cite{adsr:1997,depp:2023,gass:2002,gass:2004,hoch:2017,trms:2007,vaen:2021,widt:2022}; see \cite{nast:2023} for some recent applications. 
In many of  these applications,  only discretized {data samples} are available; even for a continuum domain  $\Omega$ enclosed by a simple Jordan curve in which $f$ is analytic, by the maximum modulus  principle, we can first sample $f$ on the boundary of $\Omega$, and then compute the rational/polynomial minimax approximant of $f$ through  solving a discrete rational/polynomial minimax problem. For these cases, denote by $\{(x_j,f_j)\}_{j=1}^m$ the sampled data from  $f_j=f(x_j)\in \bbC$ ($x_j\in \Omega$) over distinct nodes  ${\cal X}=\{x_j\}_{j=1}^m$, and denote by $\bbP_n$ the set of complex polynomials with degree less than or equal to $n$. We consider the following discrete  rational  approximation problem
\begin{equation}\label{eq:bestf0}
\inf_{\xi=p/q\in\scrR_{(n_1,n_2)}}\|\Bf-\xi(\bx)\|_{\infty},
\end{equation}
where $\scrR_{(n_1,n_2)}:=\{{p}/{q}|p\in \bbP_{n_1},~0\not\equiv q\in\bbP_{n_2} \}$, 
$\Bf=[f_1,\dots,f_m]^{\T}\in\bbC^m ~(n_1+n_2+2\le m)$,  $\bx=[x_1,\dots,x_m]^{\T}\in\bbC^m$, $\xi(\bx)=[\xi(x_1),\dots,\xi(x_m)]^{\T}\in\bbC^m$, and
\begin{equation}\nonumber 
\|\Bf-\xi(\bx)\|_{\infty}:=\max_{1\le j\le m}\left|f_j-\frac{p(x_j)}{q(x_j)}\right|.
\end{equation}
In case when the infimum of \eqref{eq:bestf0} is attainable,  we call the function $\xi^*={p^*}/{q^*}\in \scrR_{(n_1,n_2)}$  from 
\begin{equation}\label{eq:bestf}
{p^*}/{q^*} \in \arg\min_{\xi=p/q\in\scrR_{(n_1,n_2)}}\|\Bf-\xi(\bx)\|_{\infty},
\end{equation}
the rational minimax approximant  \cite{tref:2019a} of  $f(x)$ over ${\cal X}$.

In general, computing the discrete rational minimax approximation is much more challenging than the polynomial (i.e., $n_2=0$)  minimax problem. Indeed, for the polynomial case, it is guaranteed that there is a unique minimax approximant which can be characterized by a necessary and sufficient condition (e.g., \cite{rish:1961} and \cite[Theorem 2.1]{will:1972});  for the rational case, the infimum of \eqref{eq:bestf0} may not be achievable, and {even if it is}, there may be multiple minimax approximants \cite{natr:2020,tref:2019a}. Furthermore, local best approximants may exist \cite{natr:2020,tref:2019a}.  Necessary conditions for the rational minimax approximant have been developed in, e.g.,  \cite{elli:1978,gutk:1983,rutt:1985,sava:1977,this:1993,will:1972,will:1979,wulb:1980}, and Ruttan  \cite[{Theorem} 2.1]{rutt:1985} contributes a sufficient condition.

Lawson's iteration \cite{laws:1961}  is a classical and effective method for computing the discrete linear (polynomial) minimax approximant.  The idea of Lawson is to approximate the minimax polynomial $p^*$ (i.e., $n_2=0$ in \eqref{eq:bestf}) by a sequence of polynomials $\{p^{(k)}\}$, each as the solution of the weighted least-squares problem:
$$
p^{(k)}=\argmin_{p\in \bbP_n}\sum_{j=1}^mw_j^{(k)}|f_j-p(x_j)|^2,
$$
where $\bw^{(k)}=[w_1^{(k)},\dots,w_m^{(k)}]^{\T}\in {\cal S}$ is the weight vector at the $k$th iteration in the  probability simplex:
\begin{equation}\nonumber 
{\cal S}:=\{\bw=[w_1,\dots,w_m]^{\T}\in \bbR^m: \bw\ge 0 ~{\rm and } ~\bw^{\T}\be=1\},~~\be=[1,1,\dots,1]^{\T}.
\end{equation}
At the $k$th step in  Lawson's iteration, it  updates the element of the weight vector as 
\begin{equation}\label{eq:lawsonstep}
w_j^{(k+1)}=\frac{w_j^{(k)}\left|f_j-p^{(k)}(x_j)\right|^{\beta}}{\sum_{i}w_i^{(k)}\left|f_i-p^{(k)}(x_i)\right|^{\beta}},~~\forall j,
\end{equation} 
where $\beta>0$ is the so-called {\it Lawson exponent} and is originally set as $\beta=1$ (see e.g., \cite{clin:1972,laws:1961}). Due to its relation with weighted least-squares problems, Lawson's iteration is an iteratively reweighted least-squares (IRLS) iteration. 
Convergence analysis and some variants have been discussed (e.g., \cite{bama:1970,clin:1972,coop:2007,hoka:2020,loeb:1957,sako:1963,yazz:2023}).

However, extension of Lawson's iteration for the rational minimax approximation problem with both computationally high efficiency and theoretical guarantee is nontrivial and challenging. Particularly, for computing the minimax rational {solution} $\xi^*=p^*/q^*$ of \eqref{eq:bestf}, two difficulties related with a basic Lawson's step  \eqref{eq:lawsonstep} are: how to define a suitable approximation $\xi^{(k)}=p^{(k)}/q^{(k)}$ associated with the current weight vector $\bw^{(k)}$, and how to choose a suitable Lawson exponent $\beta$ for convergence?

Some versions of Lawson's iteration have been discussed for the rational minimax approximation problem \eqref{eq:bestf0}. For example, the Loeb algorithm  \cite{loeb:1957} (the same method was also proposed in \cite{sako:1963} known as the SK iteration by Sanathanan and  Koerner) uses the reciprocal of the current denominator $q^{(k)}(x_j)$ as weights and compute the approximation $\xi^{(k)}=p^{(k)}/q^{(k)}$ from a weighted linearization associated with the current weights. 
A recent work \cite{hoka:2020} further improves the basic SK iteration by proposing a stabilized SK iteration.  
Another remarkable work {in the rational approximation literature} is the adaptive Antoulas-Anderson (AAA) {algorithm} \cite{nase:2018} and its extension, {the AAA-Lawson algorithm} \cite{fint:2018,natr:2020,nast:2023}. AAA represents the rational approximation in barycentric form and selects the associated support points iteratively in an adaptive way for stability; in AAA-Lawson \cite{fint:2018,natr:2020,nast:2023}, the algorithm further introduces weights updated according to \cite[Equ. (8.5)]{fint:2018}
\begin{equation}\label{eq:lawsonstep-aaalawson}
w_j^{(k+1)}=\frac{w_j^{(k)}\left|f_j-\xi_{\tiny \rm AAA-Lawson}^{(k)}(x_j)\right|^{\beta}}{\sum_{i}w_i^{(k)}\left|f_i-\xi_{\tiny \rm AAA-Lawson}^{(k)}(x_i)\right|^{\beta}},~~\forall j,
\end{equation} 
where $\xi_{\tiny \rm AAA-Lawson}^{(k)}$ is from a weighted linearization \cite[Equ. (8.4)]{fint:2018} associated with the current weight vector $\bw^{(k)}$ (see also \eqref{eq:AAA-Lawson}). Other versions of Lawson's iteration can be found in \cite{bama:1970,coop:2007}. However, 
to our best knowledge, no convergence guarantee has been established for these versions of Lawson's iteration in theory, and in some cases, the computed rational approximants can be local best or near-best \cite{fint:2018,natr:2020}.  Indeed, as remarked for AAA-Lawson in \cite{fint:2018} that {\it ``its convergence is far from understood, and even when it does converge, the rate is slow (linear at best)''}, and {\it ``convergence analysis appears to be highly nontrivial"}. 

In this paper, we shall establish the convergence of a version of Lawson's iteration (Algorithm \ref{alg:Lawson}), namely, {\tt d-Lawson}, proposed recently in \cite{zhyy:2023}. In the same Lawson's updating fashion for the weights $w_j^{(k+1)}$ given in \eqref{eq:lawsonstep-aaalawson}, the main difference between  {\tt d-Lawson} and AAA-Lawson lies in defining the associated approximant $\xi^{(k)}=p^{(k)}/q^{(k)}$, and we shall  clarify this difference in detail in section \ref{subsec:convg-framework}.  For {\tt d-Lawson}, \cite{zhyy:2023} reveals that it  can be viewed as a method for solving the dual problem 
$
\max_{\bw\in {\cal S}} d_2(\bw)
$
of the original rational minimax approximation problem. The weight $w_j$ is the corresponding dual variable associated with the node $x_j$.  Such a dual problem is guaranteed to obtain the original minimax solution $\xi^*$ under Ruttan's sufficient condition  (\cite[{Theorem} 2.1]{rutt:1985};  see also \cite[Theorem 2]{isth:1993} and \cite[Theorem 3]{this:1993});  moreover, numerically, {\tt d-Lawson} was observed to converge monotonically \cite{zhyy:2023} with respect to the dual objective function $d_2(\bw)$. The framework for handling the rational minimax approximation problem \eqref{eq:bestf0}  in  \cite{zhyy:2023} can be well described by Figure \ref{fig:framework} (see \cite[Figure 1]{zhyy:2023}). It should be pointed out that {\tt d-Lawson} reduces to the classical Lawson's iteration \cite{laws:1961} for the linear minimax approximation problem if $n_2=0$.   

\begin{figure}[htb!!!]  
 \tikzstyle{P} = [rectangle, rounded corners, minimum width=1cm, minimum height=1cm, text centered, text width=3.6cm, draw=black,  fill=red!30, drop shadow]
\tikzstyle{D} = [rectangle, rounded corners, minimum width=2cm, minimum height=0cm, text centered, text width=3.42cm, draw=black, fill=blue!30, drop shadow] 
\tikzstyle{M} = [rectangle, rounded corners, minimum width=2cm, minimum height=0cm, text centered, text width=3.42cm, draw=black, fill=white!, drop shadow] 

\tikzstyle{arrow} = [ -->, >=stealth]
\tikzstyle{linepart} = [draw, thick, color=black!50, -latex', dashed] 
\tikzstyle{line} = [draw,  thick,    color=red!,  -latex']
 
\newcommand{\Primal}[2]{node (p#1) [P] {#2}}
\newcommand{\Dual}[2]{node (p#1) [D] {#2}}
    \newcommand{\Method}[2]{node (p#1) [M] {#2}}
\begin{tikzpicture}[node distance=1cm,x=0.675cm,y=0.6cm]\nonumber
\path \Primal{2}{{\small Primal problem: \vskip -7mm$${  (Chebyshev ~appr.)~} $$ \vskip -10mm $${\rm min-max} $$}}; 
\path (p2.east)+(7.5,-6.0) \Dual{3}{{\small Dual problem ${\rm (max-min)}$: \vskip -7mm $$\max_{\bw\in {\cal S}}d_2(\bw)$$}};
  \path (p3.east)+(0,6) \Method{5}{{\small \vskip -8mm$${\rm Lawson's~ iteration}$$ }};  
  \draw [line,red] (p5.south)  -- node [above=1pt, sloped] {solve the dual} (p3);
\draw [line,red] (p2.east)  -- node [above=5pt, sloped] {weak duality $\checkmark$} (p3);
\draw [line,red] (p2.east)   -- node [below=-1pt, sloped] {strong duality under} (p3);
\draw [line,red] (p2.east)  -- node [below=9pt, sloped] {Ruttan's condition} (p3);              
\draw (2.5,-1 ) to [bend left=-30pt,  below=0pt,draw=graphicbackground,sloped] node [midway] {\small Lagrange duality on a linearization} (7.7,-6 );
\end{tikzpicture}
\caption{Framework \cite[Figure 1]{zhyy:2023} of the dual programming and Lawson's iteration for the rational minimax approximation of \eqref{eq:bestf}.}
\label{fig:framework}
 \end{figure}
For this version {\tt d-Lawson} of Lawson's iteration  (Algorithm \ref{alg:Lawson}),  we shall contribute theoretical convergence analysis  for both the linear and rational minimax approximation problems. In particular, we show that
\begin{itemize}
\item[ (i)] for the linear minimax approximation problem,  $\beta=1$ is a near-optimal Lawson exponent in Lawson's iteration \eqref{eq:lawsonstep}, and 
\item[(ii)] for the rational minimax approximation {problem, under certain conditions}, {\tt d-Lawson} converges monotonically with respect to the dual objective function $d_2(\bw)$ for any sufficiently small $\beta>0$ in \eqref{eq:lawsonstep-aaalawson}; moreover, the convergent pair $(\bw,\xi)$ {satisfies} the following complementary slackness {conditions}: 
\begin{equation}\label{eq:complement0}
w_j|f_j-\xi(x_j)|\left(|f_j-\xi(x_j)|-c\right)=0, \quad\forall j=1,2,\dots,m,
\end{equation}
for some constant $c$. 
The complementary slackness means that any node associated with positive weight either is an interpolation point or has a constant error.
\end{itemize}
These theoretical findings lay a solid ground for {\tt d-Lawson} (Algorithm \ref{alg:Lawson}).

We organize the paper as follows. In section \ref{sec_dual}, we first introduce the dual problem associated with a linearization of the rational minimax approximation problem  \eqref{eq:bestf0} {using} Lagrange duality \cite{zhyy:2023}. These results are mainly from \cite{zhyy:2023}. {We shall introduce the concept of strong duality and explore its connection with Ruttan's sufficient condition}; {an optimality} condition for computing the dual objective function and {\tt d-Lawson} \cite{zhyy:2023} (Algorithm \ref{alg:Lawson})  for solving  the dual problem will be presented; furthermore, a roadmap of our convergence analysis and a comparison between AAA-Lawson and {\tt d-Lawson} are given in this section. In section \ref{sec:cong_Lawson}, we shall establish an important lower bound for the dual objective function value $d_2(\bw^{(k+1)})$ based on the information at the $k$th iteration.  Relying on this lower bound, we prove in section \ref{sec:beta_Lawson} that for the linear minimax approximation problem, there is a $\beta_0>1$ so that for any $\beta\in (0,\beta_0)$, Lawson's iteration converges monotonically, and $\beta=1$ is the near-optimal Lawson exponent. In section \ref{sec:beta_Lawson_rational},  we will use the lower bound to show  that monotonic convergence of {\tt d-Lawson} occurs  generically for any sufficiently small $\beta>0$, and furthermore, we shall prove the complementary slackness whenever $d_2(\bw^{(k+1)})=d_2(\bw^{(k)})$. Finally concluding remarks  are drawn in section \ref{sec_conclude}.

{\bf Notation}.
We follow the notation in \cite{zhyy:2023} in this paper. The imaginary unit is ${\tt i}=\sqrt{-1}$, and for  $\mu= \mu^{\tt r}+ {\tt i} \mu^{\tt i}\in \bbC$, we denote its modulus $|\mu|=\sqrt{(\mu^{\tt r})^2+(\mu^{\tt i})^2}$ and its conjugate $\bar\mu= \mu^{\tt r}- {\tt i} \mu^{\tt i}$, where  ${\rm Re }(\mu)=\mu^{\tt r}\in \bbR$ and ${\rm Im }(\mu)=\mu^{\tt i}\in \bbR$  are the real and imaginary part of $\mu$, respectively. Bold lower case letters are used to represent column vectors,  and ${\mathbb C}^{n\times m}$ (resp. ${\mathbb R}^{n\times m}$) stands for the set
of all $n\times m$ complex (resp. real) matrices, with the identity $I_n\equiv [\be_1,\be_2,\dots,\be_n]\in\bbR^{n\times n}$, where $\be_i$ is its $i$th column with $i\in \{1,2,\dots,n\}$. For a vector $\bx\in \bbC^n$,  $\diag(\bx)=\diag(x_1,\dots,x_n)$ is the diagonal matrix, and $\|\bx\|_\alpha=(\sum_{j=1}^n|x_j|^\alpha)^\frac1\alpha$ is the vector $\alpha$-norm ($\alpha\ge 1$) of $\bx$. For  $\bx,\by\in \bbC^n$ with $y_j\ne 0,~1\le j\le n$, we define $\bx./\by=[x_1/y_1,\dots,x_n/y_n]^{\T}.$   
For  a matrix $A\in\bbC^{m\times n}$,  ${\rm span}(A)$ represents the column space of $A$; $A^{\HH}$ (resp. $A^{\T}$)  and $A^{\dag}$ are the conjugate transpose (resp. transpose) and the Moore-Penrose inverse of $A$, respectively. We also adopt MATLAB-like convention to represent the sub-matrix $A({\cal I}_1,{\cal I}_2)$ of $A$, consisting of intersections of rows and columns indexed by ${\cal I}_1\subseteq \{1,2,\dots, m\}$ and ${\cal I}_2\subseteq \{1,2,\dots, n\}$, respectively. 

\section{A dual problem and Lawson's iteration} \label{sec_dual} 
First, let $$\bbP_{n_1}={\rm span}(\psi_0(x),\dots,\psi_{n_1}(x)) ~{\rm and} ~\bbP_{n_2}={\rm span}(\phi_0(x),\dots,\phi_{n_2}(x))$$ be the chosen bases for the numerator and denominator  polynomial spaces, respectively, and we write   $p/q\in \scrR_{(n_1,n_2)}$ as 
\begin{equation} \label{eq:paramt_ab}
\frac{p(x)}{q(x)}=\frac{[\psi_0(x),\dots,\psi_{n_1}(x)] \ba}{[\phi_0(x),\dots,\phi_{n_2}(x)] \bb},~~\mbox{for ~some~} \ba\in \bbC^{n_1+1}, ~\bb\in \bbC^{n_2+1}.
\end{equation}  
For the given  ${\cal X}=\{x_j\}_{j=1}^m$ with $|{\cal X}|=m$,  we have the basis matrix for $p\in \bbP_{n_1}$:
\begin{equation}\nonumber 
\Psi=\Psi(x_1,\dots,x_m;n_1):=\left[\begin{array}{cccc}\psi_0(x_1) & \psi_1(x_1) & \cdots & \psi_{n_1}(x_1) \\\psi_0(x_2) & \psi_1(x_2) & \cdots & \psi_{n_1}(x_2)  \\ \vdots & \cdots & \cdots& \vdots  \\\psi_0(x_m) & \psi_1(x_m) & \cdots & \psi_{n_1}(x_m) \end{array}\right],~\Psi_{i,j}=\psi_{j-1}(x_i),
\end{equation}
and analogously, we have $\Phi=\Phi(x_1,\dots,x_m;n_2)=[\phi_{j-1}(x_i)] \in \bbC^{m\times (n_2+1)}$.

For a given irreducible   $\xi(x)=p(x)/q(x)\in\scrR_{(n_1,n_2)}$, if $|\xi(x)|$ is bounded for any $x\in {\cal X}$, then it is easy to see that $q(x)\ne 0$. We define the maximum error
\begin{equation}\label{eq:exi}
e(\xi):=\max_{x\in {\cal X}}|f(x)-\xi(x)|=\|\Bf-\xi(\bx)\|_\infty.
\end{equation}
The {\it defect}  of an irreducible $\xi(x)=p(x)/q(x)\in\scrR_{(n_1,n_2)}$ is 
\begin{equation}\label{eq:defect}
\upsilon(p,q):=\min(n_1-\deg(p),n_2-\deg(q)),
\end{equation} 
where $\deg(p)$ and $\deg(q)$ are  the degrees of $p$ and $q$, respectively.  When $\upsilon(p,q)=0$, we say $\xi(x)=p(x)/q(x)$ is {\it non-degenerate}.
As a necessary condition for the irreducible minimax rational approximant $\xi^*=p^*/q^*$  of \eqref{eq:bestf},  the following result  states that the number of {\it reference points} ({aka} the extreme points), i.e., nodes $x_j\in {\cal X}$ that {achieve} $$\left|f_j- \frac{p^*(x_j)}{q^*(x_j)}\right|=e(\xi^*),$$  is at least $n_1+n_2+2-\upsilon(p^*,q^*)$. 
\begin{theorem}{\rm (\cite[Theorem 2.5]{gutk:1983})}\label{thm:extremalNo}
Given $m\ge {n_1}+n_2+2$ distinct nodes ${\cal X}=\{x_j\}_{j=1}^m$ on $\Omega$, suppose $\xi^*=p^*/q^*\in \scrR_{(n_1,n_2)}$ is an irreducible rational polynomial and denote the extremal set   $ {\cal  X}_e(\xi^*)\subseteq {\cal X}$ by 
\begin{equation}\label{eq:extremalset}
 {\cal  X}_e(\xi^*):=\left\{x_j\in {\cal X}:\left|f_j-\frac{p^*(x_j)}{q^*(x_j)}\right|=e(\xi^*)\right\}. 
\end{equation}
If $\xi^*$ is a solution to \eqref{eq:bestf} with 
$
\eta_\infty=\|\Bf-\xi^*(\bx)\|_\infty,
$
 then the cardinality  $|{\cal  X}_e(\xi^*)|\ge  n_1+n_2+2-\upsilon(p^*,q^*)$; that is,  ${\cal  X}_e(\xi^*)$ contains at least $n_1+n_2+2-\upsilon(p^*,q^*)$ nodes. 
\end{theorem}

\subsection{A linearization}\label{subsec:linearity}
In \cite{zhyy:2023}, by introducing a real variable $\eta$, the original minimax problem \eqref{eq:bestf0}  is transformed into the following linearization:
  \begin{align}\nonumber
&\inf_{\eta\in \bbR,~p\in \bbP_{n_1},~{0\not \equiv q}\in \bbP_{n_2}}\eta \\\label{eq:linearity}
 s.t., ~& |f_jq(x_j)-p(x_j)|^{{{2}}}\le \eta |q(x_j)|^{{{2}}},~~\forall j=1,2,\dots,m.
\end{align}
Unlike the original {bi-level} min-max problem \eqref{eq:bestf0},  problem \eqref{eq:linearity} is a single-level optimization, whose infimum is also attainable.
\begin{theorem}\label{thm:attainable_equi}
Given $m\ge {n_1}+n_2+2$ distinct nodes ${\cal X}=\{x_j\}_{j=1}^m$ on $\Omega\subset\bbC$, let $\eta_2$ be the infimum of \eqref{eq:linearity}. Then $\eta_2$  is attainable by a pair $(p,q)$ with $p\in \bbP_{n_1}$ and $0\not \equiv q\in \bbP_{n_2}$. 
\end{theorem}
\begin{proof}
For any  $0\not \equiv q\in \bbP_{n_2}$, as $m\ge {n_1}+n_2+2$, there exists $x_j\in \cal{X}$ with $q(x_j)\ne 0$. This implies that $0\le \sqrt{\eta_2} \le f_{\max}:=\max_{1\le j\le m}|f_j|$ where $\sqrt{\eta_2}  \le f_{\max}$ is obtained by choosing a feasible solution $ {\eta} =(f_{\max})^2, ~p\equiv 0$ and $q\equiv 1$. Suppose  $\{(\eta^{(k)},p^{(k)},q^{(k)})\}$ is the feasible sequence so that $p^{(k)}\in \bbP_{n_1}$, $0\not \equiv q^{(k)}\in \bbP_{n_2}$ and $\eta^{(k)}\rightarrow \eta_2$ as $k\rightarrow \infty$.  Noting that $\{(\eta^{(k)},\tau^{(k)}p^{(k)},\tau^{(k)}q^{(k)})\}$ is also feasible for \eqref{eq:linearity} for any $\tau^{(k)}\ne 0$, we then can choose $\tau^{(k)}$ so that the coefficient vector $\bb^{(k)}=[b_0^{(k)},\dots,b_{n_2}^{(k)}]^{\T}\in \bbC^{n_2+1}$ of $\tau^{(k)}q^{(k)}(x)=\sum_{i=0}^{n_2}b_i^{(k)}x^i$ satisfies  $\|\bb^{(k)}\|_2=1 ~(\forall k\ge 0)$. Let $\tilde p^{(k)}:=\tau^{(k)}p^{(k)}$ and $\tilde q^{(k)}:=\tau^{(k)}q^{(k)}$. Thus $\{\tilde q^{(k)}\}$ has a convergent subsequence resulting from a convergent subsequence of $\{\bb^{(k)}\}$. For simplicity of presentation, we assume that $\{\tilde q^{(k)}\}$ itself converges to $0\not \equiv  q\in \bbP_{n_2}$, and $\eta^{(k)}\le \rho_0 ~(\forall k\ge 0)$ for some $\rho_0>0$.

Note that $|f_j\tilde q^{(k)}(x_j)-\tilde p^{(k)} (x_j)|\le \sqrt{\eta^{(k)}} |\tilde q^{(k)}(x_j)|$ implies 
$$
|\tilde p^{(k)} (x_j)|\le \sqrt{\eta^{(k)}} |\tilde q^{(k)}(x_j)|+ |f_j\tilde q^{(k)}(x_j)|\le  (\sqrt{\rho_0}+ {f_{\max}})\rho_1 =:\tilde \rho_1,~\forall x_j\in {\cal X},
$$
where $\rho_1=\sum_{i=0}^{n_2}|x_J|^i$ with $|x_J|=\max_{1\le j\le m}|x_j|$. Denoting $\tilde p^{(k)}(x)=\sum_{i=0}^{n_1}a_i^{(k)}x^i$ in the monomial basis, we can choose the first $n_1+1$ nodes $\{x_j\}_{j=1}^{n_1+1}$ to have 
$$[\tilde p^{(k)}(x_1),\dots,\tilde p^{(k)}(x_{{n_1}+1})]^{\T}=V(x_1,\dots,x_{n_1+1})[a_0^{(k)},\dots,a_{n_1}^{(k)}]^{\T} $$
where $V(x_1,\dots,x_{n_1+1})$ is the Vandermonde matrix associated with $\{x_j\}_{j=1}^{n_1+1}$. This gives 
$$\|[a_0^{(k)},\dots,a_{n_1}^{(k)}]\|_2\le \tilde \rho_1 \sqrt{n_1+1}\|[V(x_1,\dots,x_{n_1+1})]^{-1}\|_2,~\forall k\ge 0.$$ Therefore,    $\{\tilde p^{(k)}\}$ also has a convergent subsequence   with a limit polynomial $p\in \bbP_{n_1}$, and consequently, $(\eta_2, p,q)$ solves \eqref{eq:linearity}.
\end{proof}

As a further step, we shall show in Theorem \ref{thm:linearity} that  the two infimums of \eqref{eq:bestf0} and \eqref{eq:linearity} are consistent. Thus if \eqref{eq:bestf0} admits a solution, it can be recovered from the linearization \eqref{eq:linearity}. It is a generalization of \cite[Theorem 2.1]{zhyy:2023}. The next lemma is used for proving Theorem \ref{thm:linearity}.
\begin{lemma}\label{lem:interp}
Let $1\le s\le \max(n_1,n_2)$. Given any distinct nodes $x_1,\dots,x_s$ in $\bbC$ and the corresponding $f_j\in \bbC, ~j=1,\dots,s$, for any $\epsilon>0$, there is a $\frac{p}{q} \in \scrR_{(n_1,n_2)}$ such that 
$\left|f_j-\frac{p(x_j)}{q (x_j)}\right|\le \epsilon$ and $q (x_j)\ne 0$  $\forall j=1,\dots, s$.
\end{lemma}
\begin{proof}
For $s\le n_1$, choose $q\equiv 1$ and $p\in \bbP_{n_1}$ satisfying the interpolation conditions $p(x_j)=f_j ~(1\le j\le s)$; thus   $\left|f_j-\frac{p(x_j)}{q (x_j)}\right|=0,~1\le j\le s$. For $n_1<s\le n_2$, let $p\equiv 1$ and $q \in \bbP_{n_2}$ satisfying the interpolation conditions: $q(x_j)=\frac{1}{f_j}$ if $f_j\ne 0$, and $q(x_j)=\frac{2}{\epsilon}$ otherwise; thus   $\left|f_j-\frac{p(x_j)}{q (x_j)}\right|\le \frac{\epsilon}{2}<\epsilon,~1\le j\le s$. This completes the proof.
\end{proof}

\begin{theorem}\label{thm:linearity}
 Given $m\ge {n_1}+n_2+2$ distinct nodes ${\cal X}=\{x_j\}_{j=1}^m$ on $\Omega\subset\bbC$, let ${\eta_2}$ be the  infimum of \eqref{eq:linearity}. Then $\eta_2= (\eta_\infty)^2$. Furthermore, whenever \eqref{eq:bestf0} has an irreducible solution $\xi^*=p^*/q^*$,   $((\eta_\infty)^{{{2}}}, p^*,q^*)$ is a solution of  \eqref{eq:linearity}. 
\end{theorem}
\begin{proof} 
First, it is true that ${\sqrt{\eta_2}}\le \eta_\infty$. In fact, for any irreducible sequence $\{\xi^{(k)}=p^{(k)}/q^{(k)}\}$ satisfying $\eta_{\infty}^{(k)}:=\|\Bf-\xi^{(k)}(\bx)\|_\infty\rightarrow \eta_\infty$,  $((\eta_{\infty}^{(k)})^2,p^{(k)},q^{(k)})$ is feasible for \eqref{eq:linearity}, and thus ${\sqrt{\eta_2}}\le \eta_{\infty}^{(k)}\rightarrow \eta_\infty$. 

To show ${\sqrt{\eta_2}}= \eta_\infty$, suppose by contradiction that ${\sqrt{\eta_2}}< \eta_\infty$. {From Theorem \ref{thm:attainable_equi},  \eqref{eq:linearity} has a solution $({\eta_2},\hat p,\hat q)$  with $\hat q\not \equiv 0$.} Suppose without loss of generality that  $\hat q(x_j)=0$ for $j=1,\dots,s$. As $\hat q\not \equiv 0$, we have   $s\le n_2.$ The constraints of \eqref{eq:linearity} lead to 
 $$|f_j\hat q(x_j)-\hat p(x_j)|^{{2}}=|\hat p(x_j)|^{{2}}\le  0\Longrightarrow  \hat p(x_j)=0,\quad \forall j=1,2,\dots,  s.$$
The proof is completed if  $s=0$ (i.e., $\hat q(x_j)\ne 0~\forall j =1,2,\dots, m$) because 
$$
|f_j \hat q(x_j)-\hat p(x_j)|^{{2}}\le {\eta_2} |\hat q(x_j)|^{{2}}\Longrightarrow \left|f_j -\frac{\hat p(x_j)}{\hat q(x_j)}\right|\le {\sqrt{\eta_2}}<\eta_\infty,
$$
contradicting with the fact that $\eta_\infty$ is the infimum of \eqref{eq:bestf0}. 

We next consider $s\ge 1$.  Choose a parameterized rational polynomial 
{$$
\frac{\hat p(x;\delta)}{\hat q(x;\delta)}=\frac{\hat p(x)+\delta \cdot p(x)}{\hat q(x)+\delta \cdot q(x)}\in  \scrR_{(n_1,n_2)},~~\delta \in \bbC,
$$
for some polynomials $p\in \bbP_{n_1}$ and $q\in \bbP_{n_2}$.} We will finish the proof by showing that there is a pair $(p,q)$ and sufficiently small $\delta$ so that $\hat q(x_j;\delta)\ne 0~ \forall j=1,\dots,m,$ and 
\begin{equation}\label{eq:contraction}
\max_{1\le  j\le m}\left|f_j-\frac{\hat p(x_j;\delta)}{\hat q(x_j;\delta)}\right|<\eta_\infty,
\end{equation}
which contradicts with the fact that $\eta_\infty$ is the infimum of \eqref{eq:bestf0}. 

To this end, by Lemma \ref{lem:interp}, for $0<\epsilon:= \frac{\eta_\infty+\sqrt{\eta_2}}{2}<\eta_{\infty}$, there is a $\frac{p}{q} \in \scrR_{(n_1,n_2)}$ such that 
$\left|f_j-\frac{p(x_j)}{q (x_j)}\right|\le \epsilon$ and $q (x_j)\ne 0,~\forall j=1,\dots, s$. 
For any sufficiently small $\delta\ne 0$, we know that $\hat q(x_j;\delta)\ne 0,~1\le j\le m$. Moreover, the parameterized $\frac{\hat p(x;\delta)}{\hat q(x;\delta)}$ with this pair $(p,q)$ satisfies
$$
\left|\frac{\hat p(x_j;\delta)}{\hat q(x_j;\delta)}-f_j\right|=\left|\frac{\hat p(x_j)+\delta \cdot p(x_j)}{\hat q(x_j)+\delta \cdot q(x_j)}-f_j\right|=\left|\frac{p(x_j)}{q(x_j)}-f_j\right|\le \epsilon< \eta_\infty,~\forall  j=1,\dots,s.
$$
On the other hand, for any $j=s+1,\dots,m$, we have 
$$
|f_j \hat q(x_j)-\hat p(x_j)|^{{2}}\le {\eta_2} |\hat q(x_j)|^{{2}}\Longrightarrow \left|f_j -\frac{\hat p(x_j)}{\hat q(x_j)}\right|\le \sqrt{\eta_2}<\eta_\infty,
$$
and
$$
\left|\frac{\hat p(x_j)+\delta \cdot p(x_j)}{\hat q(x_j)+\delta \cdot q(x_j)}-\frac{\hat p(x_j)}{\hat q(x_j)}\right|=|\delta|\cdot \left|\frac{ p(x_j)\hat q(x_j)- q(x_j)\hat p(x_j)}{(\hat q(x_j)+\delta  \cdot q(x_j))\hat q(x_j)}\right|\rightarrow 0,~~{\rm as}~~\delta\rightarrow 0.
$$
Consequently,  for any sufficiently small $\delta$,  it follows 
\begin{align*}
\left|\frac{\hat p(x_j;\delta)}{\hat q(x_j;\delta)}-f_j\right|&\le \left|\frac{\hat p(x_j;\delta)}{\hat q(x_j;\delta)}-\frac{\hat p(x_j)}{\hat q(x_j)}\right|+\left|f_j -\frac{\hat p(x_j)}{\hat q(x_j)}\right|<\eta_\infty, ~\forall j=s+1,\dots, m.
\end{align*}
This leads to \eqref{eq:contraction}, and we have $\sqrt{\eta_2}= \eta_\infty$ by contradiction.

For the last part, according to $\sqrt{\eta_2}= \eta_\infty$ and Theorem \ref{thm:attainable_equi},  whenever \eqref{eq:bestf0} admits an irreducible  solution $\xi^*=p^*/q^*$,  the triplet $((\eta_\infty)^{{2}},p^*,q^*)$ is   a solution of  \eqref{eq:linearity}. 
\end{proof} 

Theorem \ref{thm:linearity} reveals a profound connection, specifically $\eta_2 = (\eta_\infty)^2$, linking the   rational minimax approximation problem \eqref{eq:bestf0} to its linearized counterpart in \eqref{eq:linearity}.  Despite this significant relationship, there are instances where the attainability of infimum  $\eta_\infty$ of \eqref{eq:bestf0} remains elusive and unattainable. 
\begin{example}\label{Eg:eg1}
Consider $n_1=n_2=1$. Let $x_j=j~(1\le j\le 4)$,  $f_1=0$ and $f_j=1$ for $j=2,3,4$. It is true that for any $\xi\in \scrR_{(1,1)}$, $\max_{1\le j\le 4}|f_j-\xi(x_j)|>0$ because a nonconstant $\xi\in \scrR_{(1,1)}$ can only take each value once. However, the infimum $\eta_\infty=0$ since for the sequence $\xi^{(k)}(x)=\frac{x-1}{x-1+\frac1k}\in \scrR_{(1,1)}$, it holds that 
$$
\max_{1\le j\le 4}|f_j-\xi^{(k)}(x_j)|\rightarrow 0,~{\rm as}~k\rightarrow \infty.
$$
On the other hand, for \eqref{eq:linearity}, the infimum $\eta_2=0$ is attainable by $p(x)=q(x)=x-1$.
\end{example}

\subsection{A dual problem}\label{subsec:dual}
Even though we have transformed the original bi-level min-max problem \eqref{eq:bestf0} into \eqref{eq:linearity}, directly handling  \eqref{eq:linearity} is still hard. The idea in \cite{zhyy:2023} is to develop the dual problem of \eqref{eq:linearity}, and then employ Lawson's idea for the linear Chebyshev approximation problem for solving the dual problem. {In particular, the dual function of \eqref{eq:linearity}  can be given by  \cite{zhyy:2023}}
\begin{align}\nonumber
d_2(\bw)&=\min_{\begin{subarray}{c}p\in \bbP_{n_1},~q\in \bbP_{n_2}\\
            \sum_{j=1}^m w_j |q(x_j)|^2=1\end{subarray}}\sum_{j=1}^m w_j |f_j q(x_j)-p(x_j)|^2\\\label{eq:rat-d-compt}
            &=\min_{\begin{subarray}{c}  \ba\in \bbC^{n_1+1},~ \bb\in \bbC^{n_2+1}\\
            \|\sqrt{W}\Phi \bb\|_2 =1\end{subarray}}\left\|\sqrt{W}[-\Psi,F\Phi]  \left[\begin{array}{c}\ba \\\bb\end{array}\right]\right\|_2^2,
\end{align}
where  $W=\diag(\bw), ~F = \diag(\Bf)$. The following weak duality \cite{zhyy:2023} has been proved: 
\begin{equation}\label{eq:weakduality}
\forall \bw\in {\cal S}, ~d_2(\bw)\le (\eta_\infty)^2\Longrightarrow  \max_{\bw\in {\cal S}}d_2(\bw)\le  (\eta_\infty)^2.
\end{equation}
Moreover, by relying on Ruttan's sufficient condition (\cite[{Theorem} 2.1]{rutt:1985};  see also \cite[Theorem 2]{isth:1993} and \cite[Theorem 3]{this:1993}), we can also have the following theoretical guarantee  (see \cite[Theorems 2.2,   4.3]{zhyy:2023}) for solving the minimax approximant $\xi^*$.

\begin{theorem}\label{thm:q-dual}
 Given $m\ge {n_1}+n_2+2$ distinct nodes ${\cal X}=\{x_j\}_{j=1}^m$ on $\Omega$, we have  the weak duality \eqref{eq:weakduality}.  Let $\bw^*\in {\cal S}$ be the solution to the dual 
 problem 
 \begin{equation}\label{eq:dual}
\max_{\bw\in {\cal S}} d_2(\bw),
\end{equation} 
and {$(\ba(\bw^*),\bb(\bw^*))$} achieve the minimum $d_2(\bw^*)$ of \eqref{eq:rat-d-compt}. Then if the associated rational polynomial  $\xi^*(x)=\frac{p^*(x)}{q^*(x)}=\frac{[\psi_0(x),\dots,\psi_{n_1}(x)] {\ba(\bw^*)}}{[\phi_0(x),\dots,\phi_{n_2}(x)] {\bb(\bw^*)}}$ is irreducible and 
\begin{equation}\label{eq:strongdual}
\|\Bf-\xi^*(\bx)\|_\infty=\sqrt{d_2(\bw^*)},
\end{equation}
then $\xi^*$ is the minimax approximant of \eqref{eq:bestf0}. When \eqref{eq:strongdual} holds, then we also have the following complementary slackness property:
\begin{equation}\label{eq:complementopt}
w_j^*\left(\|\Bf-\xi^*(\bx)\|_\infty -|f_j-\xi^*(x_j)|\right)=0, \quad\forall j=1,2,\dots,m.
\end{equation}
\end{theorem}

We remark that the condition \eqref{eq:strongdual} implies the strong duality $\max_{\bw\in {\cal S}}d_2(\bw)=  (\eta_\infty)^2$ \cite[Theorem 4.3]{zhyy:2023}, which is equivalent to Ruttan's sufficient condition for the original rational minimax problem \eqref{eq:bestf0}. Therefore, under Ruttan's sufficient condition, the complementary slackness property \eqref{eq:complementopt} necessarily holds at the maximizer $\bw^*$ of  \eqref{eq:rat-d-compt}.  Furthermore, in the framework of dual programming  \eqref{eq:rat-d-compt}, the accuracy of the associated  approximation $\xi$  corresponding to the minimization \eqref{eq:rat-d-compt} at an approximation $\bw$ of $\bw^*$, can be measured by the relative error
 \begin{equation}\nonumber 
 \epsilon(\bw):=\left|\frac{\sqrt{d_2(\bw)}-e(\xi)}{e(\xi)}\right|.
\end{equation} 
This serves as a stopping rule for Lawson's iteration (Algorithm \ref{alg:Lawson}). 

\subsection{Optimality for the dual objective function}\label{subsec:optimality}
To compute the dual function $d_2(\bw)$, a minimization problem \eqref{eq:rat-d-compt} needs to be solved. The following proposition    provides the optimality condition for this minimization. 
\begin{proposition} (\cite[Proposition 3.1]{zhyy:2023}) \label{prop:dual_GEP}
For $\bw\in {\cal S}$, we have  
\begin{itemize}
\item[(i)]
$\bc(\bw)=\left[\begin{array}{c} \ba(\bw) \\ \bb(\bw)\end{array}\right] \in \bbC^{n_1+n_2+2}$ is a solution of \eqref{eq:dual} if and only if it {is} an eigenvector of the Hermitian positive semi-definite generalized eigenvalue problem $(A_{\bw},B_{\bw})$ and  $d_2(\bw)$ is the smallest eigenvalue  satisfying
\begin{equation}\label{eq:dual_GEP}
A_{\bw} \bc(\bw)=d_2(\bw) B_{\bw} \bc(\bw)~ \mbox{ and }~ \bc(\bw)^{\HH}B_{\bw}\bc(\bw) =1,
\end{equation} 
where
\begin{align*}\nonumber
A_{\bw}:&=[-\Psi,F\Phi]^{\HH}W[-\Psi,F\Phi]=\left[\begin{array}{cc}\Psi^{\HH}W\Psi & -\Psi^{\HH}FW\Phi \\-\Phi^{\HH} WF^{\HH}\Psi & \Phi^{\HH}F^{\HH}WF\Phi\end{array}\right],\\\nonumber
B_{\bw}:&= [0,\Phi]^{\HH}W [0,\Phi]=\left[\begin{array}{cc}0 & 0 \\0 & \Phi^{\HH}W\Phi \end{array}\right];
\end{align*}
\item[(ii)] the  Hermitian matrix $H_{\bw}:=A_{\bw} -d_2(\bw) B_{\bw} \succeq 0$, i.e., $H_{\bw}$ is  positive semi-definite;

\item[(iii)] let $W^{\frac12}\Phi=Q_qR_q$ and $W^{\frac12}\Psi=Q_pR_p$ be the thin QR factorizations where $Q_q\in \bbC^{m\times \wtd n_2}$, $Q_p\in \bbC^{m\times \wtd n_1}$, $R_q\in \bbC^{ \wtd n_2 \times (n_2+1)}$,  $R_p\in \bbC^{\wtd n_1 \times (n_1+1)}$ with $\wtd n_1={\rm rank}(W^{\frac12}\Psi)$ and $\wtd n_2={\rm rank}(W^{\frac12}\Phi)$. Then $(d_2(\bw),R_q\bb(\bw))$ is an eigenpair associated with the smallest   eigenvalue of the   Hermitian positive semi-definite  matrix $S_F-S_{qp}S_{qp}^{\HH}\in \bbC^{\wtd n_2\times \wtd n_2}$
satisfying 
\begin{subequations}\label{eq:qeig}
\begin{align}\label{eq:qeiga}
(S_{F}-S_{qp}S_{qp}^{\HH})R_q\bb(\bw)&=d_2(\bw) R_q\bb(\bw),~~\| R_q\bb(\bw)\|_{2}=1,\\\label{eq:qeigb}
R_p\ba(\bw)&=S_{qp}^{\HH}R_q\bb(\bw),
\end{align}
\end{subequations}
where 
$S_F=Q_q^{\HH} |F|^2Q_q\in \bbC^{\wtd n_2\times \wtd n_2},~S_{qp}=Q_q^{\HH}F^{\HH} Q_p\in \bbC^{\wtd n_2 \times \wtd n_1}.$
Moreover, letting  $[Q_p,Q_p^{\perp}]\in \bbC^{m\times m}$ be unitary, then $(\sqrt{d_2(\bw)},R_q\bb(\bw))$ is the right singular pair associated with the smallest singular value of both $(Q_p^{\perp})^{\HH}FQ_q\in \bbC^{(m-\wtd n_1-1)\times (\wtd n_2+1)}$ {and $(I_m-Q_pQ_p^{\HH})FQ_q\in \bbC^{m\times (\wtd n_2+1)}$.}
\end{itemize}
\end{proposition}

As remarked in \cite{zhyy:2023},  the normalized condition $\bc(\bw)^{\HH}B_{\bw}\bc(\bw) =1$ in \eqref{eq:dual_GEP}  can always be fulfilled whenever $q\not \equiv 0$ and there are at least $n_2+1$ positive elements in $w_j$. In fact,  $0=\bc(\bw)^{\HH}B_{\bw}\bc(\bw) =\sum_{j=1}^mw_j|q(x_j)|^2\Longrightarrow  w_j q(x_j)=0, ~{\forall j=1,2,\dots,m};$ thus,  any node $x_j$ with $w_j>0$ is a zero of $q$, and if $\bw$ has at least $n_2+1$ positive elements, it implies that $q  \equiv 0$. Thus,  in the following discussion, we assume that  {any iterate $\bw^{(k)}$ has at least} $n_2+1$ positive elements.

We remark further  that, by rewriting the first $n_1+1$ rows and the last $n_2+1$ rows of the optimality condition $A_{\bw} \bc(\bw)=d_2(\bw) B_{\bw} \bc(\bw)$ in \eqref{eq:dual_GEP}, we have 
\begin{corollary}\label{cor:optimality}
Let $ \bp=\Psi \ba(\bw)$ and $\bq=\Phi \bb(\bw)$ be from the solution of \eqref{eq:rat-d-compt} with the weight vector $\bw$. Then 
\begin{equation}\label{eq:optimalitytwo}
F\bq-\bp\perp_{\bw}{\rm span}(\Psi),~~F^{\HH}(F\bq-\bp)-d_2(\bw)\bq\perp_{\bw}{\rm span}(\Phi).
\end{equation}
\end{corollary}
Besides the optimality in solving the minimization  \eqref{eq:rat-d-compt} for the minimum $d_2(\bw)$, we can further {obtain} the gradient of $d_2(\bw)$ with respect to the dual variable $\bw$.
\begin{proposition}(\cite[Proposition 5.1]{zhyy:2023})\label{prop:grad}
For $\bw>0$,  let  $d_2(\bw)$ be the smallest eigenvalue of the Hermitian positive semi-definite generalized eigenvalue problem \eqref{eq:dual_GEP}, and $\bc(\bw)=\left[\begin{array}{c} \ba(\bw) \\ \bb(\bw)\end{array}\right] \in \bbC^{n_1+n_2+2}$ be the associated eigenvector. Denote $$\bp =[p_1,\dots,p_m]^{\T}=\Psi \ba(\bw)\in \bbC^m ~{\rm and}~\bq  =[q_1,\dots,q_m]^{\T}=\Phi \bb(\bw)\in \bbC^m.$$ 
If $d_2(\bw)$ is a simple eigenvalue, then 
$d_2(\bw)$ is differentiable with respect to $\bw$ and its gradient is 
\begin{equation} 
\nabla d_2(\bw)=\left[\begin{array}{c}|f_1 q_1 - p_1 |^2-d_2(\bw)| q_1 |^2 \\|f_2 q_2 - p_2 |^2-d_2(\bw)| q_2 |^2  \\\vdots \\|f_m q_m - p_m |^2-d_2(\bw)| q_m |^2 \end{array}\right] =:|F \bq -\bp |^2-d_2(\bw)|\bq |^2\in \bbR^m. \label{eq:gradd2}
\end{equation}
\end{proposition}

\subsection{The {\tt d-Lawson} iteration}
Within the framework of the dual problem, it has been claimed  \cite{zhyy:2023} that Lawson's iteration is a method for solving the dual problem \eqref{eq:dual}.  For the rational minimax approximation problem, Lawson's iteration \cite{zhyy:2023} is implemented as in Algorithm \ref{alg:Lawson}. Numerical results have been reported in \cite{zhyy:2023}, indicating that  {\tt d-Lawson}\footnote{The MATLAB code of {\tt d-Lawson} is available at \url{https://ww2.mathworks.cn/matlabcentral/fileexchange/167176-d-lawson-method}.}  generally converges monotonically with respect to the dual function value $d_2(\bw)$.  In our following discussion, we assume $\epsilon_w=0$ in Step 1 of Algorithm \ref{alg:Lawson}.

\begin{algorithm}[thb!!!]
\caption{A rational Lawson's iteration ({\tt d-Lawson}) \cite{zhyy:2023} for \eqref{eq:bestf}} \label{alg:Lawson}
\begin{algorithmic}[1]
\renewcommand{\algorithmicrequire}{\textbf{Input:}}
\renewcommand{\algorithmicensure}{\textbf{Output:}}
\REQUIRE Given samples $\{(x_j,f_j)\}_{j=1}^m$ and $0\le n_1+n_2+2\le m$ with $x_j\in \Omega$, a relative tolerance for the strong duality $\epsilon_r>0$, the maximum number $k_{\rm maxit}$ of iterations; 
        \smallskip
\STATE  (Initialization) Let $k=0$; choose $0<\bw^{(0)}\in {\cal S}$  and a  tolerance $\epsilon_w$ for the weights;
\STATE (Filtering) Remove nodes $x_i$ with $w_i^{(k)}<\epsilon_w$;
\STATE Compute  $d_2(\bw^{(k)})$ and the associated vector $\xi^{(k)}(\bx)={[\Psi \ba(\bw^{(k)})]./[\Phi \bb(\bw^{(k)})]}$  according to Proposition \ref{prop:dual_GEP};
  
\STATE ({Stopping} rule)  Stop either if $k\ge k_{\rm maxit}$ or 
\begin{equation}\nonumber 
\epsilon(\bw^{(k)}):=\left|\frac{\sqrt{d_2(\bw^{(k)})}-e(\xi^{{(k)}})}{e(\xi^{{(k)}})}\right|<\epsilon_r,~~{\rm where}~~e(\xi^{{(k)}})=\|\Bf-\xi^{(k)}(\bx)\|_\infty;
\end{equation} 
 
\STATE (Updating weights) Update the weight vector $\bw^{(k+1)}$ according to 
\begin{equation}\label{eq:lawson}
w_j^{(k+1)}=\frac{w_j^{(k)}\left|f_j-\xi^{(k)}(x_j)\right|^{\beta}}{\sum_{i}w_i^{(k)}\left|f_i-\xi^{(k)}(x_i)\right|^{\beta}},~~\forall j,
\end{equation} 
with the Lawson exponent $\beta>0$, and {go to} Step 2 with $k=k+1$.
\end{algorithmic}
\end{algorithm}  

\begin{remark}\label{rmk:lineara}
It is interesting to point out that {\tt d-Lawson} in Algorithm \ref{alg:Lawson} naturally reduces to the classical Lawson's iteration \cite{laws:1961} for the linear (polynomial) minimax approximation problem when $n_2=0$, which corresponds to $q\equiv 1$ and $\bq^{\HH}W\bq\equiv 1$.  This observation unifies  our subsequent convergence analysis. 
\end{remark}
\begin{remark}\label{rmk:linearb}
In practice, the implementation of Step 3 should handle the stability and accuracy for  computing $d_2(\bw^{(k)})$ and the associated vector $\xi^{(k)}(\bx)=\bp^{(k)}./\bq^{(k)}$ {where $\bp^{(k)}=\Psi \ba(\bw^{(k)})$ and $\bq^{(k)}=\Phi \bb(\bw^{(k)})$.} In \cite{zhyy:2023}, the Vandermonde with Arnoldi process \cite{brnt:2021,hoka:2020,zhsl:2023} is employed for this step.
\end{remark}

It is noticed that Lawson's updating \eqref{eq:lawson} relies on the error vector $\Bf-\xi^{(k)}(\bx)$ to update the weights. In this procedure, each entry of the denominator vector $\bq^{(k)}$ is assumed to be nonzero; this is generally the case in practice,  
and in our following discussion,  we make the following two assumptions:

\begin{itemize}
\item[({\bf A1})] $q_j^{(k)}=\be_j^{\T}\bq^{(k)}\ne 0$ for any {$j=1,2,\dots,m$} and any $k\ge 0$;
\item[({\bf A2})] the cardinality $|{\cal I}_{\bw^{(k)}}|\ge \max\{n_1+1,n_2+1\}$, where $${\cal I}_{\bw^{(k)}}:=\{j|w_j^{(k)}\ne 0, ~{1\le j\le m}\}.$$
\end{itemize}
\begin{proposition}\label{prop:Iwk}
If $d_2(\bw^{(k)})>0$, then 
$$
{\cal I}_{\bw^{(k+1)}}=\{j|w_j^{(k)}r_j^{(k)}\ne 0, ~{1\le j\le m}\},~~{\rm where}~~r_j^{(k)}=|f(x_j)-\xi^{(k)}(x_j)|.
$$
\end{proposition}
\begin{proof}
By ({\bf A1}), the condition $d_2(\bw^{(k)})>0$ implies $\sum_{i}w_i^{(k)}\left|f_i-\xi^{(k)}(x_j)\right|^{\beta}>0$. Thus, according to Lawson's updating rule \eqref{eq:lawson}, the conclusion follows due to the fact that $w_j^{(k+1)}=0$ is equivalent to $w_j^{(k)}r_j^{(k)}=0$.
\end{proof}

\subsection{The roadmap of the convergence analysis}\label{subsec:convg-framework}
Before proceeding the detailed convergence analysis for {\tt d-Lawson}, we present a roadmap of our analysis in Fig.  \ref{fig:roadmap}  which can help the reader to follow the overly technical presentation and make the main novelty of our theoretical analysis clearer. 
  \begin{figure}[htb!!!] 
  \begin{tikzpicture}[node distance=20pt]
  \node[draw, rounded corners]                        (primal)   {\small Minimax approx. \eqref{eq:bestf0}};
    \node[draw,right=80pt, rounded corners]                        (dual)   {\small The dual \eqref{eq:dual}: $\max_{\bw\in {\cal S}} d_2(\bw)$};
  \node[draw, below=of dual]                         (optimality)  {\small Corollary \ref{cor:optimality}: optimality for $d_2(\bw^{(k+1)})$};
  \node[draw, below=of optimality]  (bound)  {\small Theorem \ref{thm:bnd1}: a bound $\sqrt{d_2^{(k+1)}}\ge d_2^{(k)}\cdot \chi^{(k)}(\beta)$};
    \node at (2,-4)  [draw,  rounded corners, aspect=1.2]     (linear)  {\small Linear minimax approx.}; 
  \node at (8,-4) [draw,  rounded corners, aspect=1.2]     (rational)  {\small Rational minimax approx.};

  \node at (-0,-6)  [draw,  aspect=1.2]  (L-decrease)     {\small $ \chi^{(k)}(\beta)\ge \sqrt{\frac{1}{d_2^{(k)}}}$};
    \node at (3.4,-5.9)  [draw,  aspect=1.2]  (optimal)     {\small $1=\argmax_{\beta>0}\chi^{(k)}(\beta)$};
    \node  [draw,  fill=red!10,draw=blue, below=30pt of L-decrease]  (linear-decrease)     {{\small $d_2^{(k+1)}\ge d_2^{(k)}$}};
      \node  [draw, fill=red!10,draw=blue,  below=38pt of optimal]  (beta=1)     {{\small $\beta=1$ is optimal}};
    
  \node at (8,-6)  [draw,  aspect=1.2]  (R-decrease)     {\small $\chi^{(k)}(\beta)\ge \sqrt{\frac{1}{d_2^{(k)}}}$};
    \node  [draw, fill=red!30,draw=blue, below=30pt of R-decrease]  (r-decrease)     {{\small $d_2^{(k+1)}\ge d_2^{(k)}$}};
      \node at (1.8,-7.5)  [draw, fill=red!30,draw=blue, below=30pt of r-decrease]  (complement)     {{\tiny complementary slackness \eqref{eq:complement0}}};
        \node at (1,-8)  [draw, fill=red!30,draw=blue, left=90pt of complement]  (r-best)     {{\tiny $\xi^{(k)}$ is the minimax approx. of \eqref{eq:bestf0}}};
 
  \draw[->] (primal)  -- (dual)[left];
  \draw[->] (dual)  --node [right=-1pt] {{\tiny at the $(k+1)^{th}$ iteration}}  (optimality);
  \draw[->] (optimality) -- (bound);
  \draw[->] (bound) -- (linear);
  \draw[->] (bound) -- (rational);
  \draw[->] (L-decrease) -- node [right=17pt] {\small Proposition \ref{prop:min}} (linear-decrease);
    \draw[->] (linear) --  node [left=5pt] {\tiny \underline{{$\forall \beta\in [0,2]$}}}(L-decrease);
    
       \draw[->] (rational) --  node [left=-2pt] {\tiny \underline{$\forall \beta>0$ sufficiently small}}(R-decrease);
        \draw[->] (linear) -- (optimal);
      \draw[->] (L-decrease) -- (linear-decrease);
           \draw[->] (optimal) -- (beta=1);
           
            \draw[->] (R-decrease) -- node [left=-1pt] {\small  Theorem \ref{thm:mont-rational}} (r-decrease);
            \draw[->] (r-decrease) -- node [left=-1pt] {\tiny  \underline{if $d_2^{(k+1)}= d_2^{(k)}$} } (complement);
    \draw[->] (R-decrease) --  (r-decrease);
 \draw[->] (complement) -- node [above=-1pt] {\tiny  under additional conditions } (r-best);
  \draw[->] (complement) -- node [below=-1pt] {\tiny  (iii) of Theorem \ref{thm:mont-rational} } (r-best);
\end{tikzpicture}  \caption{The roadmap of the convergence analysis for Lawson's iteration, where $d_2^{(k)}=d_2(\bw^{(k)})$}\label{fig:roadmap}
 \end{figure}

From this roadmap in Fig. \ref{fig:roadmap}, we remark that the main idea is to establish the monotonic property $d_2({\bw^{(k+1)}})\ge d_2({\bw^{(k)}})$ with respect to the dual objective function. In this framework, the dual problem \eqref{eq:dual} plays the most important role, whose solution  corresponds to the minimax approximant of the original \eqref{eq:bestf0} under Ruttan's sufficient condition (see Theorem \ref{thm:q-dual}). The remarkable point is that  {\tt d-Lawson} serves as a simple yet efficient monotonically ascend method for the dual \eqref{eq:dual}.

 Based on the roadmap in Fig. \ref{fig:roadmap}, we also make a remark on the AAA-Lawson iteration  \cite{fint:2018}. Comparing with {\tt d-Lawson}, the main differences between these two   methods are: (1) AAA-Lawson uses   barycentric representations for $\xi^{(k)}$ and,  (2) for a given $\bw^{(k)}\in {\cal S}$, they compute the associated $\xi^{(k)}$ in different ways. We remark that barycentric representations can be very useful for numerical stability, but mathematically (i.e., in exact arithmetic) the dual framework does not rely on a particular representation for $\xi^{(k)}$; in fact, barycentric representations  can also be adopted in representing \eqref{eq:paramt_ab} and the {\tt d-Lawson} version with barycentric representations is one of our ongoing works. The crucial difference between AAA-Lawson and {\tt d-Lawson} is the way of computing the new rational approximant $\xi^{(k)}$ at a given $\bw^{(k)}$; in particular, AAA-Lawson defines $\xi^{(k)}$ as 
 \begin{equation}\label{eq:AAA-Lawson}
\xi_{\tiny{\rm  AAA-Lawson}}^{(k)}=\frac{p^{(k)}}{q^{(k)}},~ (p^{(k)}, q^{(k)})=\argmin_{\begin{subarray}{c} p\in \bbP_{n_1}, q\in \bbP_{n_2} \\ \|\ba\|_2^2+\|\bb\|_2^2=1\end{subarray}}\sum_{j=1}^m w^{(k)}_j |f_j q(x_j)-p(x_j)|^2,
\end{equation}
where $\ba$ and $\bb$ are coefficient vectors of $p^{(k)}$ and $q^{(k)}$, respectively,  in barycentric representations,   
while {\tt d-Lawson} defines $\xi^{(k)}$ as
 \begin{align}\nonumber
\xi^{(k)}&=\frac{p^{(k)}}{q^{(k)}},~ (p^{(k)}, q^{(k)})=\argmin_{\begin{subarray}{c}p\in \bbP_{n_1},~q\in \bbP_{n_2}\\
            \sum_{j=1}^m w^{(k)}_j |q(x_j)|^2=1\end{subarray}}\sum_{j=1}^m w^{(k)}_j |f_j q(x_j)-p(x_j)|^2.
\end{align}
It appears that $\xi_{\tiny{\rm  AAA-Lawson}}^{(k)}$ is a seeming extension of the linear minimax approximation {problem} to the rational case which {lacks} an underlying dual problem. When returning to the linear case with $n_2=0$,  both reduce {to} Lawson's original form (see  Remark \ref{rmk:lineara}), but for the rational case $n_2>0$, $\xi^{(k)}$ is the {more powerful} formulation and generalization because it is derived from the dual framework. We refer the reader to \cite{zhyy:2023} for  {a comparison of the numerical performance of these two methods in terms of accuracy and robustness.} 
 
\section{A lower bound of the dual objective function value}\label{sec:cong_Lawson}
We first relate Lawson's iteration with the gradient ascent direction. Indeed, when the current $0<\bw\in {\cal S}$, $q_j=\bq^{\T}\be_j\ne 0~{\forall j=1,2,\dots,m}$ and $d_2(\bw)$ is the simple eigenvalue of $(A_{\bw},B_{\bw})$, then we know that the direction 
\begin{equation}\nonumber 
\bg(\bw):=\diag\left(\frac{w_1}{|q_1|^2},\ldots,\frac{w_m}{|q_m|^2}\right)\nabla d_2(\bw)\in \bbR^m
\end{equation}
is an ascent direction for the dual function $d_2(\bw)$ because by \eqref{eq:gradd2}
$$
\bg(\bw)^{\T}\nabla d_2(\bw)=\nabla d_2(\bw)^{\T}\diag\left(\frac{w_1}{|q_1|^2},\ldots,\frac{w_m}{|q_m|^2}\right)\nabla d_2(\bw)>0.
$$
Note that using this direction, the {update with step size} $\mu=\frac{1}{d_2(\bw)}>0$ gives
\begin{equation}\nonumber 
\wtd\bw:=\bw+\mu\bg=\diag(\bw)|\Bf-\xi(\bx)|^2/d_2(\bw)
\end{equation}
which, after a scaling $\wtd \bw\leftarrow\frac{\wtd \bw}{\be^{\T}\wtd \bw}\in {\cal S}$,  implies the iteration is the same as Lawson's iteration \eqref{eq:lawson} with $\beta=2$. Note that the scaling $\wtd \bw\leftarrow\frac{\wtd \bw}{\be^{\T}\wtd \bw}$ can be viewed as a certain projection onto ${\cal S}$.  From this point of view, {\it we can say that a Lawson's iteration in Algorithm \ref{alg:Lawson} with the Lawson exponent $\beta=2$ is just an ascent gradient step with a specific step-size followed by a certain projection onto ${\cal S}$.}

To more clearly see the relation between the two consecutive objective values $d_2(\bw^{(k)})$ and $d_2(\bw^{(k+1)})$, we introduce the $\bw$-inner (positive semidefinite) product defined by $\langle\by,\bz \rangle_{\bw}=\by^{\HH}W\bz$ and $\|\by\|_{\bw}=\sqrt{\by^{\HH}W\by}$, where $W=\diag(\bw)$ and $\bw\ge0$. The following lemma, which is a generalized result of the standard least-squares problem, plays an important role in establishing the convergence of Lawson's iteration \cite{laws:1961} for the linear minimax approximation problem \cite[Lemma 13-12]{rice:1969}. 
\begin{lemma}\label{lem:LSmax}
Given $0\le \bw\in \bbR^m$ satisfying $|{\cal I}_{\bw}|\ge n$ with ${\cal I}_{\bw}=\{j|w_j>0\}$,  a matrix $A\in \bbC^{m\times n}$ with ${\rm rank}(A({\cal I}_{\bw},:))=n$ and $\bz\in \bbC^m$, let $\bx_*$ be the solution to the least-squares problem 
$$
\min_{\bx\in {\rm span}(A)}\|\bx-\bz\|_{\bw}.
$$
Then we have 
$$
\frac{\bx_*-\bz}{\|\bx_*-\bz\|_{\bw}}\in \arg\max_{\|\by\|_{\bw}=1,~\by^{\HH}WA=0}|\langle\bz,\by\rangle_{\bw}|.
$$
 \end{lemma}
\begin{proof}
Let $W=\diag(\bw)$.  We first consider the case $\bw>0$. 

Denote by $P\in \bbC^{m\times (m-n)}$ the $\bw$-orthogonal basis for the complement \cite[Chapter 1.6.3]{kato:1966} of  ${\rm span}(A)$ satisfying   ${\rm span}(P)\oplus {\rm span}(A)=\bbC^m$, $P^{\HH}WP=I_{m-n}$ and $P^{\HH}WA=0$. Note that $PP^{\HH}W$ and  $I-PP^{\HH}W$ are  projections onto ${\rm span}(P)$ and ${\rm span}(A)$, respectively.
For the least-squares problem, we know that $\bx_*$ is the solution if and only if
$$
(\bx_*-\bz)^{\HH}WA=0,~~(i.e., ~\bx_*-\bz= PP^{\HH}W \bz).
$$
Note $\|\bx_*-\bz\|_{\bw}=\|PP^{\HH}W \bz\|_{\bw}=\|P^{\HH}W\bz\|_2$. 
Also,  the constraint $\|\by\|_{\bw}=1,~\by^{\HH}WA=0$ can be parameterized as $\by=P\bt$ for $\bt\in \bbC^{m-n}$ with $\|\bt\|_{2}=1$. Since
$$
|\langle\bz,\by\rangle_{\bw}|=| \bz^{\HH}WP\bt |\le \|P^{\HH}W\bz\|_2,
$$
and the equality holds if $\bt=P^{\HH}W\bz/\|P^{\HH}W\bz\|_2$, it implies that
$$\by=\frac{PP^{\HH}W\bz}{\|P^{\HH}W\bz\|_2}=\frac{\bx_*-\bz}{\|\bx_*-\bz\|_{\bw}}$$
 is  the maximizer and the conclusion follows.
 
For the general case $\bw\ge 0$, assume ${\cal I}_{\bw}=\{1,2,\dots,t\}$, and partition accordingly  $\bx=[\bx_1^{\T},\bx_2^{\T}]^{\T}$, $\bz=[\bz_1^{\T},\bz_2^{\T}]^{\T}$, $\bw=[\bw_1^{\T},\bw_2^{\T}]^{\T}$,  $\by=[\by_1^{\T},\by_2^{\T}]^{\T}$, and $A=[A_1^{\T},A_2^{\T}]^{\T}$ with $A_1\in \bbC^{t\times n}$ and ${\rm rank}(A_1)=n$. Note 
 \[
 \min_{\bx\in {\rm span}(A)}\|\bx-\bz\|_{\bw}=\min_{\bx_1\in {\rm span}(A_1)}\|\bx_1-\bz_1\|_{\bw_1},
 \]
and the solution {$\bx_*$} is unique by assumptions. Indeed, {$\bx_*$} satisfies 
$
(\bx_*-\bz)^{\HH}WA=0.
$

On the other hand, 
\[
\max_{\|\by\|_{\bw}=1,~\by^{\HH}WA=0}|\langle\bz,\by\rangle_{\bw}|=\max_{\|\by_1\|_{\bw_1}=1,~\by_1^{\HH}W_1A_1=0}|\langle\bz_1,\by_1\rangle_{\bw_1}|
\]
and by the proof of the first case $\bw>0$, we know that any maximizer $\by=[\by_1^{\T},\by_2^{\T}]^{\T}$ takes $\by_1=\frac{\bx{_1}_*-\bz_1}{\|\bx{_1}_*-\bz_1\|_{\bw_1}}$ and $\forall\by_2\in \bbC^{m-t}$. Thus, the conclusion follows.
\end{proof}

\begin{lemma}\label{lem:perp}
Let $(\ba(\bw^{(k)}), \bb(\bw^{(k)}))$ be the solution and  $d_2(\bw^{(k)})$ be the minimum of \eqref{eq:rat-d-compt} with the weight $\bw^{(k)}\in {\cal S}$. Denote  $\bp^{(k)}=\Psi \ba(\bw^{(k)}),~\bq^{(k)}=\Phi \bb(\bw^{(k)})$ and $\xi^{(k)}(\bx)=\bp^{(k)}./\bq^{(k)}$.
If $d_2(\bw^{(k)})>0$, then 
\begin{equation}\label{eq:perp}
(F\bq^{(k)}-\bp^{(k)})./|\Bf-\xi^{(k)}(\bx)|^{ {\beta}}\perp_{\bw^{(k+1)}} {\rm span}(\Psi).
\end{equation}
Here we set $\frac{0}{0}=0$ for convenience.
\end{lemma}
\begin{proof}
For simplicity, we denote the pairs $(\bp^{(k)},\bq^{(k)})$ and $(\bp^{(k+1)},\bq^{(k+1)})$ by $(\bp,\bq)$ and $(\wtd \bp,\wtd \bq)$ at the $k$th and $(k+1)$th step, respectively;  {this applies to other quantities}. 

Notice that $d_2(\bw)>0$ implies $${\cal I}_{\wtd \bw}=\{j|\wtd w_j\ne 0, ~{1\le j\le m}\}=\{j|w_jr_j\ne 0, ~{1\le j\le m}\}$$ by Proposition \ref{prop:Iwk}.  Also, according to \eqref{eq:lawson}, 
$\wtd w_j=w_j|f_j-\xi_j|^\beta/\gamma_\beta,$
where 
$$\gamma_{\beta}:=\sum_{j=1}^mw_j \left|f_j-\xi (x_j)\right|^\beta\left(=\left\||\Bf-\xi (\bx)|^{\frac{\beta}{2}}\right\|_{\bw}^2\right)>0.
$$

As $(\wtd \ba, \wtd\bb)$ is the solution to \eqref{eq:rat-d-compt} associated with the weight $\wtd\bw$, for the given $\wtd \bb$, the vector $\wtd \ba$ is the solution to the following least-squares problem
\begin{equation}\nonumber 
\min_{  \ba\in \bbC^{n_1+1} }\left\| F\wtd\bq-\Psi\ba \right\|_{\wtd\bw}=\min_{  \bz\in{\rm span}(\Psi) }\left\| F\wtd\bq-\bz \right\|_{\wtd\bw};
\end{equation}
thus we have  the optimality $(F\wtd\bq-\wtd\bp)^{\HH}\wtd W \Psi=0$ (see \eqref{eq:optimalitytwo}); analogously, for the previous iteration, we have $(F \bq- \bp)^{\HH}  W \Psi=0$ implying that  $\forall\bt\in {\rm span}(\Psi)$, 
\begin{align}\nonumber
0=\langle \bt,F \bq- \bp \rangle_{\bw}&=\sum_{j=1}^mw_j\bar t_jq_j(f_j-\xi_j) =\sum_{j\in {\cal I}_{\wtd \bw}}w_j\bar t_jq_j(f_j-\xi_j)\\\nonumber
&{=\gamma_\beta\sum_{j\in {\cal I}_{\wtd \bw}}\wtd w_j\bar t_jq_j   (f_j-\xi_j)|f_j-\xi_j|^{ -\beta}}\\\nonumber
&=\gamma_\beta\sum_{j=1}^m\wtd w_j\bar t_jq_j   (f_j-\xi_j)|f_j-\xi_j|^{ -\beta}\quad \quad \mbox{(by ~$\wtd w_j=0~\forall j\not\in {\cal I}_{\wtd \bw}$)}\\\nonumber
&=\gamma_\beta\langle \bt, (F\bq-\bp)./|\Bf-\xi(\bx)|^{ {\beta}}\rangle_{\wtd \bw}.
\end{align}
Thus, $(F\bq-\bp)./|\Bf-\xi(\bx)|^{ {\beta}}\perp_{\wtd \bw} {\rm span}(\Psi)$.  
\end{proof}

We next provide a lower bound for the dual objective function value $d_2(\bw^{(k+1)})$. This lower bound plays a crucial role in finding the near-optimal $\beta$ for Lawson's iteration in the linear minimax approximation problem, and also is a key to establish monotonic convergence $d_2(\bw^{(k+1)})\ge d_2(\bw^{(k)})$ in  the rational minimax  approximation problem.

\begin{theorem}\label{thm:bnd1}
{Let $(\ba(\bw^{(k)}), \bb(\bw^{(k)}))$ be the solution and  $d_2(\bw^{(k)})$ be the minimum of \eqref{eq:rat-d-compt} with the weight $\bw^{(k)}\in {\cal S}$. Denote  $\bp^{(k)}=\Psi \ba(\bw^{(k)}),~\bq^{(k)}=\Phi \bb(\bw^{(k)})$ and $\xi^{(k)}(\bx)=\bp^{(k)}./\bq^{(k)}$.} Then 
\begin{equation}\label{eq:bnd1}
\sqrt{d_2(\bw^{(k+1)})}\ge d_2(\bw^{(k)})\cdot \chi^{(k)}(\beta),
\end{equation}
where $\chi^{(k)}(\beta):=\frac{\left|(\bq^{(k+1)})^{\HH}W^{(k)}\bq^{(k)}\right|}{\gamma_{\beta}~\zeta_{\beta}}$, $W^{(k)}=\diag(\bw^{(k)})$, 
\begin{align*}
\gamma_{\beta}=\left\||\Bf-\xi^{(k)}(\bx)|^{\frac{\beta}{2}}\right\|_{\bw^{(k)}}^2:=\sum_{j=1}^mw_j^{(k)}\left|f_j-\xi^{(k)}(x_j)\right|^\beta, ~
\zeta_{\beta}=\left\|\frac{(\Bf\bq^{(k)}-\bp^{(k)}).}{\left|\Bf-\xi^{(k)}(\bx)\right|^{ {\beta}}}\right\|_{\bw^{(k+1)}},
\end{align*}
and we set $\frac{0}{0}=0$ for convenience.
\end{theorem}
\begin{proof} 
Following the notation in the proof of Lemma \ref{lem:perp}, we drop the superscript $k$ in each quantity related with the $k$th iteration. The result of \eqref{eq:bnd1} is trivial if $d_2(\bw)=0$, and we assume $d_2(\bw)>0$, which, by Proposition \ref{prop:Iwk}, implies ${\cal I}_{\wtd \bw}=\{j|\wtd w_j\ne 0, ~{1\le j\le m}\}=\{j|w_jr_j\ne 0, ~{1\le j\le m}\}.$

By assumptions ({\bf A1}), ({\bf A2}), $\wtd w_j=w_j|f_j-\xi_j|^\beta/\gamma_\beta$, and the fact that any $n_1+1$ rows of $\Psi$ are linearly independent, we can apply Lemmas \ref{lem:LSmax} and \ref{lem:perp} to get
\begin{align}\nonumber
d_2(\wtd \bw)&=\langle F\wtd\bq-\wtd\bp,F\wtd\bq-\wtd\bp \rangle_{\wtd\bw}  =\left|\left\langle F\wtd\bq-\wtd\bp,F\wtd\bq \right\rangle_{\wtd\bw}\right| =\sqrt{d_2(\wtd \bw)}  ~ \left|\left\langle \frac{F\wtd\bq-\wtd\bp}{\|F\wtd\bq-\wtd\bp\|_{\wtd \bw}},F\wtd\bq \right\rangle_{\wtd\bw}\right|\\ \nonumber
&=\sqrt{d_2(\wtd \bw)} ~\max_{\|\by\|_{\wtd\bw}=1,~\by^{\HH}\wtd W\Psi=0}\left|\left\langle F\wtd\bq,\by\right\rangle_{\wtd\bw}\right|
\\\nonumber
&\ge \frac{\sqrt{d_2(\wtd \bw)}}{\zeta_\beta} ~\left|\left\langle F\wtd\bq,  (F\bq-\bp)./|\Bf-\xi(\bx)|^{ {\beta} }\right\rangle_{\wtd\bw}\right|\quad \mbox{\rm (by Lemmas \ref{lem:LSmax} and \ref{lem:perp})}\\\nonumber
&=\frac{\sqrt{d_2(\wtd \bw)}}{\zeta_\beta} ~\left|\sum_{j=1}^m\wtd w_j\overline{f_j \wtd{q_j}}(f_jq_j-p_j)/|f_j-\xi_j|^\beta\right|\\\nonumber
&=\frac{\sqrt{d_2(\wtd \bw)}}{\gamma_\beta\zeta_\beta} ~\left|\sum_{j=1}^m  w_j\overline{f_j \wtd{q_j}}(f_jq_j-p_j)\right|\\\nonumber
&=\frac{\sqrt{d_2(\wtd \bw)}}{\gamma_\beta\zeta_\beta} ~\left|\langle F\wtd \bq,F\bq-\bp\rangle_{\bw}\right|  =\frac{\sqrt{d_2(\wtd \bw)}}{\gamma_\beta\zeta_\beta}  {d_2(\bw)\left|\wtd\bq^{\HH}W\bq\right|},
\nonumber 
 \end{align}
 where the last equality $\langle F\wtd \bq, F\bq-\bp\rangle_{\bw}=d_2(\bw) \wtd\bq^{\HH}W\bq$ is due to the second optimality {formula} in \eqref{eq:optimalitytwo} for the pair $(\bp,\bq)$: 
 \begin{align*}
F^{\HH}(F\bq-\bp)-d_2(\bw)\bq\perp_{\bw}{\rm span}(\Phi)&\Longrightarrow \wtd \bb^{\HH}\Phi^{\HH}W(F^{\HH}(F\bq-\bp)-d_2(\bw)\bq)=0\\
&\Longrightarrow \langle F\wtd \bq, F\bq-\bp\rangle_{\bw}=d_2(\bw) \wtd\bq^{\HH}W\bq.
\end{align*}
The proof is complete.
\end{proof}  
\section{$\beta=1$ is near-optimal   for the linear minimax approximation problem}\label{sec:beta_Lawson}
Our strategy for defining the optimal parameter $\beta$ at the $k$th iteration  is based on the lower bound \eqref{eq:bnd1}. 
We remark that $\wtd \bp$ and $\wtd \bq$ play different roles in Lawson's iteration (Algorithm \ref{alg:Lawson}). In fact,   recalling the dual function \eqref{eq:rat-d-compt} or the optimality condition \eqref{eq:optimalitytwo},  we know that $\wtd \bp$ is essentially from a least-squares problem for the given $\wtd \bq$, which is essentially linearly dependent on the data $F,\wtd \bq,\wtd\bw$; however, as $\wtd\bq$ both {appears} in the constraint and the objective function in  \eqref{eq:rat-d-compt}, $\wtd \bq$ is related with an eigenvector (refer to \eqref{eq:qeig}) of a matrix associated with the data $F,\wtd \bp,\wtd\bw$, and therefore, is nonlinearly dependent on these data. 
For the lower bound \eqref{eq:bnd1}, it is interesting to notice that only the numerator depends on the solution pair $(\wtd\bp,\wtd\bq)$ at new $\wtd\bw$, while the denominator $\gamma_\beta\zeta_{\beta}$ is only related with the information $\bp,\bq,\bw$ at the current iteration. For the polynomial minimax approximation problem (i.e., $n_2=0$), particularly, as $\bq^{(k+1)}\equiv \bq^{(k)}$ and $(\bq^{(k+1)})^{\HH}W^{(k)}\bq^{(k)}=1$ for all $k$, the lower bound in \eqref{eq:bnd1} only depends on {$\bp^{(k)}$ and $\bw^{(k)}$}. Based on this observation, we may define  a near-optimal Lawson exponent $\beta_*^{(k)}$ as the minimizer of the lower bound    \eqref{eq:bnd1}, i.e., 
\begin{equation}\label{eq:optimalbeta}
\beta_*^{(k)}=\argmin_{\beta\in \bbR} \zeta_\beta^2\gamma_\beta^2. 
\end{equation}
Recalling $r_j=|f_j-\xi^{(k)}(x_j)|$ and $${\cal I}_{\wtd \bw}=\{j|\wtd w_j>0, ~{1\le j\le m}\}=\{j|  w_jr_j>0, ~{1\le j\le m}\}$$ in Proposition \ref{prop:Iwk}, we have
\begin{equation}\label{eq:objbeta}
\nu(\beta):=\zeta_\beta^2\gamma_\beta^2=\left(\sum_{j\in {\cal I}_{\wtd \bw}}  w_j|q_j|^2r_j^{2-\beta}\right)\left(\sum_{j\in {\cal I}_{\wtd \bw}} w_j r_j^\beta\right).
\end{equation}
\begin{proposition}\label{prop:convex}
The function $\nu(\beta)$ given in \eqref{eq:objbeta} is convex.
\end{proposition}
\begin{proof}
First,  note that the derivative of $\nu(\beta)$ is 
\begin{align*}
\nu'(\beta)=(\zeta_\beta^2\gamma_\beta^2)'&=-\left(\sum_{j\in {\cal I}_{\wtd \bw}}   w_j|q_j|^2r_j^{2-\beta}\log r_j\right)\left(\sum_{j\in {\cal I}_{\wtd \bw}} w_j r_j^\beta\right)\\
&+\left(\sum_{j\in {\cal I}_{\wtd \bw}}  w_j|q_j|^2r_j^{2-\beta}\right) \left(\sum_{j\in {\cal I}_{\wtd \bw}}w_j r_j^{\beta}\log r_j\right).
\end{align*} 
Moreover, 
{\small 
\begin{align}\nonumber
&\nu''(\beta)\\\nonumber
=&\left(\sum_{j\in {\cal I}_{\wtd \bw}}  w_j|q_j|^2r_j^{2-\beta}(\log r_j)^2\right)\left(\sum_{j\in {\cal I}_{\wtd \bw}}w_j r_j^\beta\right)+\left(\sum_{j\in {\cal I}_{\wtd \bw}}   w_j|q_j|^2r_j^{2-\beta} \right) \left(\sum_{j\in {\cal I}_{\wtd \bw}} w_j r_j^{\beta}(\log r_j)^2\right)\\\nonumber
&-2\left(\sum_{j\in {\cal I}_{\wtd \bw}}  w_j|q_j|^2r_j^{2-\beta}\log r_j\right)\left(\sum_{j\in {\cal I}_{\wtd \bw}}w_j r_j^\beta\log r_j\right)\\\nonumber
\ge &2\sqrt{\left(\sum_{j\in {\cal I}_{\wtd \bw}}  w_j |q_j|^2r_j^{2-\beta}(\log r_j)^2\right)\left(\sum_{j\in {\cal I}_{\wtd \bw}}  w_j |q_j|^2r_j^{2-\beta} \right) \left(\sum_{j\in {\cal I}_{\wtd \bw}}w_j r_j^\beta\right)\left(\sum_{j\in {\cal I}_{\wtd \bw}}w_j r_j^{\beta}(\log r_j)^2\right)}\\\nonumber
& -2\left(\sum_{j\in {\cal I}_{\wtd \bw}}  w_j |q_j|^2r_j^{2-\beta}|\log r_j|\right)\left(\sum_{j\in {\cal I}_{\wtd \bw}} w_j r_j^\beta|\log r_j|\right)\\\label{eq:cauchyineq}
\ge& 2 {\left(\sum_{j\in {\cal I}_{\wtd \bw}}  w_j \left|q_j\right|^2r_j^{2-\beta}|\log r_j|\right)  \left(\sum_{j\in {\cal I}_{\wtd \bw}}w_j r_j^{\beta}|\log r_j|\right)}  \\\nonumber
&-2\left(\sum_{j\in {\cal I}_{\wtd \bw}}  w_j |q_j|^2r_j^{2-\beta}|\log r_j|\right)\left(\sum_{j\in {\cal I}_{\wtd \bw}} w_j r_j^\beta|\log r_j|\right)\\\nonumber
=&0,
\end{align}}
\hskip-1mm where the first inequality follows by using {$a^2+b^2\ge 2 {|ab|}$}, while the second is due to $\|\bu\|_2\|\bv\|_2\ge |\bu^{\T}\bv|$. 
This implies that $\nu(\beta)$ is convex.
\end{proof}
\begin{proposition}\label{prop:min}
For the polynomial minimax approximation problem, we have 
\begin{itemize}
\item [(i)]
$\beta=1$ is the   global minimizer of \eqref{eq:optimalbeta}. In this sense, $\beta=1$ achieves the maximum of the lower bound in \eqref{eq:bnd1} and can be viewed as the near-optimal Lawson exponent in Lawson's iteration;
\item[(ii)]  {for any $\beta\in [0,2]$}, the sequence of Lawson's iteration satisfies $d_2(\bw^{(k+1)})\ge d_2(\bw^{(k)})$.
\end{itemize}
 
\end{proposition}
\begin{proof}
For this special case of $\beta=1$ and $q\equiv 1$,  we only need to notice 
\begin{align*}
\nu'(1) &=-\left(\sum_{j\in {\cal I}_{\wtd \bw}} w_j r_j \log r_j\right)\left(\sum_{j\in {\cal I}_{\wtd \bw}} w_j r_j \right) +\left(\sum_{j\in {\cal I}_{\wtd \bw}}  w_j r_j \right) \left(\sum_{j\in {\cal I}_{\wtd \bw}}w_j r_j \log r_j\right)=0,
\end{align*} 
  and thus the conclusion (i) follows from Proposition \ref{prop:convex}. 

For (ii), by the convexity of $\nu(\beta)$ (i.e., $\nu''(\beta)\ge 0$) {and $${{d_2(\bw)}}=\nu(0)=\nu(2)\ge \nu(1)=\min_{\beta\in \bbR} \nu(\beta),$$ we know that  $\forall\beta\in [0,2]$, it holds that $\nu(\beta)\le {d_2(\bw)}$,} which, according to the lower bound of $d_2(\wtd \bw)$ in \eqref{eq:bnd1}, leads to $d_2(\wtd \bw)\ge d_2(\bw)$.  
\end{proof}

Proposition \ref{prop:min} establishes the monotonic convergence and also locally the near-optimal choice of the Lawson exponent at each iteration. {When $\beta = 1$, the linear convergence of Lawson's iteration was established in \cite{clin:1972,elwi:1976}. Additionally, remedies for cases where assumption ({\bf A2}) fails were  discussed in \cite[Chapter 13]{rice:1969}, while extensions of Lawson's iteration for computing a minimax approximation in a general functional subspace spanned by a Chebyshev set were proposed in \cite{rius:1968}.}

\section{Monotonic convergence and complementary slackness for the rational minimax approximation problem}\label{sec:beta_Lawson_rational}
We now consider the convergence of {\tt d-Lawson} for the rational minimax approximation problem. Different from the linear  case, a difficulty arises from the numerator  $(\bq^{(k+1)})^{\HH}W^{(k)}\bq^{(k)}$ of the lower bound \eqref{eq:bnd1}, in which $\bq^{(k+1)}$ is also dependent on the Lawson exponent $\beta$. Explicitly expressing $\bq^{(k+1)}$ in terms of $\beta$ in general is impossible because $\bq^{(k+1)}$ is related with an eigenvector of the matrix pencil $(A_{\bw},B_{\bw})$. However, locally around $\beta=0$, it is possible to analyze the term $(\bq^{(k+1)})^{\HH}W^{(k)}\bq^{(k)}$, and therefore, the lower bound in \eqref{eq:bnd1} with respect to $\beta$. Based on this observation, we can conclude that, generically, for any sufficiently small $\beta$, the monotonic property $d_2(\bw^{(k+1)})\ge d_2(\bw^{(k)})$ holds. This convergence result is consistent with the numerical experiments of {the} AAA-Lawson iteration where it is observed  \cite{fint:2018} that ``{\it taking $\beta$ to be smaller makes the algorithm much more robust}". 
 
 To develop our convergence, we need the following lemma.
\begin{lemma}\label{lem:inequality}
Given $\bs=[s_1,\dots,s_m]^{\T}\in {\cal S}$, let $t(x)$ and $h(x)$ be strictly monotonically increasing on the interval $(a,b)$. Then for any $m$ points $x_j\in (a,b),~1\le j \le m$, we have
\begin{equation}\label{eq:inequality}
\left(\sum_{j=1}^ms_j t(x_j)\right)\left(\sum_{j=1}^ms_j h(x_j)\right)\le \sum_{j=1}^ms_j t(x_j) h(x_j);
\end{equation}
moreover, the equality in \eqref{eq:inequality} holds if and only if   $s_is_j(x_i -x_j)=0$ for all $1\le i,j\le m$.
\end{lemma}
\begin{proof}
The result relies on the following inequality:
$$
t(x_i)h(x_j)+t(x_j)h(x_i)\le t(x_i)h(x_i)+t(x_j)h(x_j),~~\forall 1\le i, j \le m.
$$
In fact, the above is equivalent to 
$
(t(x_i)-t(x_j))(h(x_j)-h(x_i))\le 0
$
which is true by assumptions on $t(x)$ and $h(x)$. The equality holds if and only if   $x_j=x_i$. Thus, 
we have 
\begin{align*}
&\left(\sum_{j=1}^ms_j t(x_j)\right)\left(\sum_{j=1}^ms_j h(x_j)\right)\\
=&\sum_{j=1}^ms_j^2 t(x_j)h(x_j)+\sum_{1\le i<j\le n}s_is_j \left(t(x_i)h(x_j)+t(x_j)h(x_i)\right)\\
\le& \sum_{j=1}^ms_j^2 t(x_j)h(x_j)+\sum_{1\le i<j\le n}s_is_j \left(t(x_i)h(x_i)+t(x_j)h(x_j)\right)\\
=&\sum_{j=1}^ms_j(s_1+\dots+s_m)t(x_j)h(x_j)=\sum_{j=1}^ms_jt(x_j)h(x_j),
\end{align*}
and the equality holds if and only if $s_is_j(x_i -x_j)=0$ for all $1\le i,j\le m$. 
\end{proof}

\begin{theorem}\label{thm:mont-rational} 
At the $k$th step of {\tt d-Lawson} (Algorithm \ref{alg:Lawson}), for $\bw^{(k)}\in {\cal S}$, assume $d_2(\bw^{(k)})$ is a simple eigenvalue of the matrix pencil $(A_{\bw^{(k)}},B_{\bw^{(k)}})$ given in \eqref{eq:dual_GEP}, and  {\rm ({\bf A1})} and {\rm ({\bf A2})} hold. {Suppose $r_j^{(k)}:=|f_j-\xi^{(k)}(x_j)|>0,~\forall j\in {\cal I}_{\bw^{(k)}}:=\{j|w_j^{(k)}\ne 0, ~{1\le j\le m}\}$.} Then 
\begin{itemize}
\item[(i)]  
there is a $\beta_0>0$ so that for any $\beta\in (0,\beta_0)$, {\tt d-Lawson} gives $d_2(\bw^{(k+1)})\ge d_2(\bw^{(k)})$;
\item[(ii)] for any sufficiently small $\beta>0$, 
\begin{equation}\label{eq:complement}
d_2(\bw^{(k+1)})= d_2(\bw^{(k)})\Longrightarrow w_j^{(k)}r_j^{(k)}\left(r_j^{(k)}-c\right)=0,~ \forall 1\le j\le m,
\end{equation} 
where  ${c=\sqrt{d_2(\bw^{(k)})}}\le \|\Bf-\xi^{(k)}(\bx)\|_{\infty};$
\item[(iii)] {for the item (ii),}   if additionally, $c\ge \max_{j\not\in{\cal I}_{\bw^{(k)}}}r_j^{(k)}$ and $\xi^{(k)}$ is irreducible, then $\xi^{(k)}$ is the minimax approximant of \eqref{eq:bestf0}. 
\end{itemize}
\end{theorem}
\begin{proof}
For simplicity, we adopt the notation in the proof of Theorem \ref{thm:bnd1} by omitting the superscripts, and  we assume  that $d_2(\bw)>0$ (the conclusions are trivial when $d_2(\bw)=0$ as $w_j r_j=0,~ \forall j =1,2,\dots,m$). 
The idea for the proof is to express  and estimate the lower bound in \eqref{eq:bnd1} using  the real parameter  $\beta$ around $\beta=0$. 

Define  $W(\beta)=\diag(w_1(\beta),\dots,w_m(\beta))$ with
\begin{align*}
w_j(\beta)=\frac{w_j r_j^\beta}{\gamma_\beta}=\frac{w_j |f_j-\xi(x_j)|^\beta}{\sum_{j=1}^mw_j  |f_j-\xi(x_j)|^\beta}
\end{align*}
for which we have $w_j(0)=w_j$ and 
$$
w_j'(0)=w_j \log r_j - w_j \sum_{i=1}^mw_i\log r_i, \quad \forall j\in {\cal I}_{\wtd \bw}.
$$

In (iii) of Proposition \ref{prop:dual_GEP}, based on assumptions ({\bf A1}) and ({\bf A2}), it is true that the matrices $Q_p(\beta), R_p(\beta), Q_q(\beta)$ and $R_q(\beta)$  in QR factorizations of  $\sqrt{W(\beta)}\Phi=Q_q(\beta)R_q(\beta)$ and $\sqrt{W(\beta)}\Psi=Q_p(\beta)R_p(\beta)$,  are all locally differentiable for sufficiently small $\beta$. Moreover, as $d_2(\bw(0))$ is a simple eigenvalue of the matrix pencil $(A_{\bw(0)},B_{\bw(0)})$ given in \eqref{eq:dual_GEP}, by \eqref{eq:qeig},  it is also a simple eigenvalue of the Hermitian matrix $S_{F}(\beta)-S_{qp}(\beta)S_{qp}(\beta)^{\HH}$ at $\beta=0$. Thus, the continuity of eigenvalues implies that the smallest eigenvalue $d_2(\bw(\beta))$ is a simple eigenvalue of  the Hermitian $S_{F}(\beta)-S_{qp}(\beta)S_{qp}(\beta)^{\HH}$  for any sufficiently small $\beta\in\bbR$. Alternatively, we can say that the  eigenspace spanned by the unit-norm eigenvector  $R_q(\beta)\bb(\bw(\beta))$ of $S_{F}(\beta)-S_{qp}(\beta)S_{qp}(\beta)^{\HH}$ corresponding to $d_2(\bw(\beta))$ is one-dimensional. Based on \cite[Chapter 2.6.2]{kato:1966}, there is a continuously differentiable normalized eigenvector $R_q(\beta)\bb(\bw(\beta))$ with respect to $\beta$ around $\beta=0$. Moreover, noting that  $\bq(\beta)=\Phi \bb(\beta)=\Phi (R_q(\beta))^{-1}R_q(\beta)\bb(\beta)$, and also that  $|\bq(\beta)^{\HH}W(0)\bq(0)|$ does not change for different choices of a unit-norm eigenvector $R_q(\beta)\bb(\bw(\beta))$, in the following discussion, we can assume $\bq(\beta)$ is continuously differentiable with respect to $\beta$ around $\beta=0$. Hence,   
$$
\bq(\beta)=\bq(0)+\beta \bq'(0)+O(\beta^2).
$$
Since $$
1=\bq(\beta)^{\HH}W(\beta)\bq(\beta)\Longrightarrow {\rm Re }(\bq'(0)^{\HH}W(0)\bq(0))=-\frac12\bq(0)^{\HH}W'(0)\bq(0),
$$
for any sufficiently small $\beta\in \bbR$, we have 
\begin{align*}
&\left|\bq(\beta)^{\HH}W(0)\bq(0)\right|\\
&\ge\left|{\rm Re }(\bq(\beta)^{\HH}W(0)\bq(0))\right|\\
&=\left|\bq(0)^{\HH}W(0)\bq(0)+\beta {\rm Re }(\bq'(0)^{\HH}W(0)\bq(0))\right|+O(\beta^2)\\
&=\left|1+\beta {\rm Re }(\bq'(0)^{\HH}W(0)\bq(0))\right|+O(\beta^2)\\
&=\left|1-\frac{\beta}{2} \bq(0)^{\HH}W'(0)\bq(0)\right| +O(\beta^2)\\
&=1-\frac{\beta}{2}\left(\sum_{j\in {\cal I}_{\wtd \bw}} |q_j|^2w_j \log r_j - \left(\sum_{j\in {\cal I}_{\wtd \bw}} |q_j|^2w_j\right) \left(\sum_{j\in {\cal I}_{\wtd \bw}}w_j\log r_j\right)\right)+O(\beta^2)\\
&=1-\frac{\beta}{2}\left(\sum_{j\in {\cal I}_{\wtd \bw}} |q_j|^2w_j \log r_j -   \sum_{j\in {\cal I}_{\wtd \bw}}w_j\log r_j \right)+O(\beta^2).
\end{align*} 

With this, we can write the lower bound  in \eqref{eq:bnd1} as 
\begin{align*}
\widehat\ell(\beta):&=\frac{|\bq(\beta)^{\HH}W(0)\bq(0)|}{\gamma_{\beta}~\zeta_{\beta}}\\
&\ge\underbrace{\frac{1-\frac{\beta}{2}\left(\sum_{j\in {\cal I}_{\wtd \bw}} |q_j|^2w_j \log r_j -   \sum_{j\in {\cal I}_{\wtd \bw}}w_j\log r_j \right)}{\sqrt{\left(\sum_{j\in {\cal I}_{\wtd \bw}}  w_j|q_j|^2r_j^{2-\beta}\right)\left(\sum_{j\in {\cal I}_{\wtd \bw}}w_j r_j^\beta\right)}}}_{=:\ell(\beta)}+O(\beta^2)\\
&=:\ell(\beta)+O(\beta^2)
\end{align*}
locally at $\beta=0$. 

For $\ell(\beta)$,  by calculation, we have (with $q_j=q_j(0),~r_j=r_j(0)$ and $w_j=w_j(0)$)
\begin{align*}
&\ell'(0)\\
=&\frac{1}{2\sqrt{(d_2(\bw(0)))^3}}\left[\sum_{j\in {\cal I}_{\wtd \bw}}w_j |q_j |^2r_j^2 \log r_j  -\left(\sum_{j\in {\cal I}_{\wtd \bw}}w_j |q_j  |^2  r_j ^2\right) \left(\sum_{j\in {\cal I}_{\wtd \bw}}w_j |q_j |^2\log r_j\right)\right]\\
\ge& 0
\end{align*} 
where the last inequality is obtained by applying Lemma \ref{lem:inequality} with $s_j=w_j |q_j  |^2$, $x_j=r_j$, $t(x)=x^2$  and $h(x)=\log  x$ on the interval $(0,\infty)$. Furthermore, if there is a pair $(i,j)$ so that ${i\in {\cal I}_{\wtd \bw}}$, ${j\in {\cal I}_{\wtd \bw}}$ and  $w_iw_j(r_i-r_j)\ne 0$, then we have $\ell'(0)>0$ by Lemma \ref{lem:inequality}. In that case, we know that there is a $\beta_0>0$ such that $\ell'(\beta)>\frac12\ell'(0)$ and $\frac{\beta}{2}\ell'(0)+O(\beta^2)> 0$ for any $\beta\in (0,\beta_0)$, implying 
$$
\widehat\ell(\beta)=\widehat\ell(0)+\int_0^\beta \widehat\ell'(t)dt+O(\beta^2)\ge \widehat\ell(0)+\frac{\beta}{2}\ell'(0)+O(\beta^2)>\widehat\ell(0).
$$
This shows that, if there is a pair $(i,j)$ so that ${i,j\in {\cal I}_{\wtd \bw}}$ and $w_iw_j(r_i-r_j)\ne 0$, then a sufficiently small $\beta>0$ leads to $d_2(\bw(\beta))> d_2(\bw(0))$; conversely, for a sufficiently small $\beta>0$,
\begin{align}\nonumber
d_2(\bw(\beta))= d_2(\bw(0))&\Longrightarrow w_jw_i(r_i-r_j)=0,~\forall    {i,j\in {\cal I}_{\wtd \bw}}\\\label{eq:equcase}
&\Longrightarrow r_j={\rm a~constant~} c,~\forall    {j\in {\cal I}_{\wtd \bw}}.
\end{align}
Observe that $c\le \max_{1\le j\le m}r_j=\|\Bf-\xi(\bx)\|_{\infty}$ and also
\begin{equation}\label{eq:lowbdc}
d_2(\bw(0))=\sum_{j=1}^mw_j |q_jr_j|^2=\sum_{j\in {\cal I}_{\wtd \bw}}w_j|r_j|^2 |q_j|^2=c^2\sum_{j\in {\cal I}_{\wtd \bw}}w_j |q_j|^2{=} c^2\sum_{j=1}^mw_j |q_j|^2=c^2,
\end{equation}  
giving ${\sqrt{d_2(\bw(0))}= c}\le  \|\Bf-\xi(\bx)\|_{\infty}.$ This proves items  (i) and (ii). 

{For (iii), if we additionally}  have $c\ge \max_{j\not\in{\cal I}_{\bw}}r_j$, then $c=\sqrt{d_2(\bw(0))}=\|\Bf-\xi(\bx)\|_{\infty}$, which, according to Theorem \ref{thm:q-dual} (\cite[Theorem 4.3]{zhyy:2023}), implies that Ruttan's sufficient condition (or, equivalently,  strong duality) for the minimax approximation problem is satisfied, and $\xi^{(k)}$ is the minimax approximant of \eqref{eq:bestf}. In this case, the result in \eqref{eq:complement} can be written as 
$$
w_j\left(r_j-\|\Bf-\xi(\bx)\|_{\infty}\right)=0,~  {1\le j\le m},
$$
which is the complementary slackness property in \eqref{eq:complementopt}.
\end{proof}

\begin{remark}
\item[(a)] We remark first that the conclusion $w_j^{(k)}r_j^{(k)}\left(r_j^{(k)}-c\right)=0~ {\forall j=1,2,\dots, m}$ in \eqref{eq:complement}  is a certain complementary slackness. Indeed, it says that for node $x_j$ with $w_j^{(k)}>0$,  either $r_j^{(k)}=0$ or $r_j^{(k)}=c$; that is, any node associated with positive weight  either is an  interpolation point or has the error $c$.

\item[(b)] 
According to our proof for Theorem \ref{thm:mont-rational}, at every iteration, the Lawson exponent  $\beta_0>0$ in Theorem \ref{thm:mont-rational} is dependent on how far the continuously differentiable normalized eigenvector $R_q(\beta)\bb(\bw(\beta))$ of the Hermitian matrix $S_{F}(\beta)-S_{qp}(\beta)S_{qp}(\beta)^{\HH}$  can be extended from $\beta=0$. Explicit formulation for $\beta_0$ is hard, but intuitively, the larger {the} gap between the smallest eigenvalue $d_2(\bw)$ and the next eigenvalue is, the larger $\beta_0$ is. Currently, the proof does not guarantee   $\beta_0= 1$, which, nevertheless, performs always  well and is a recommended value  in practice \cite{zhyy:2023}.
\end{remark}
 
\section{Conclusions}\label{sec_conclude} 
 In this paper, we have established theoretical guarantees for Lawson's iteration in solving both the linear and rational minimax approximation problems.  For the rational minimax approximation problem, our results indicate that, generically, a small Lawson exponent $\beta>0$ leads to the monotonic convergence,  and also reveal some interesting properties of {\tt d-Lawson}. These theoretical guarantees, on the one hand, explain some numerical behaviors (for example,    ``{\it taking $\beta$ to be smaller makes the algorithm much more robust}" \cite{fint:2018}), and on the other hand, provide more insights on why and how Lawson's updating scheme works. These theoretical results lay a solid ground for this version {\tt d-Lawson} of Lawson's iteration.

 \section*{Acknowledgments} {We express our gratitude to the anonymous referees for their comments and questions, which have significantly contributed to enhancing our manuscript in several ways: 1) motivating the addition of a new subsection \ref{subsec:convg-framework} to more effectively elucidate intricate technical details and clarify the   novelty of our theoretical analysis; 2) prompting us  to conduct a more in-depth reflection on the relationship between  \eqref{eq:bestf0} and \eqref{eq:linearity}; and 3) improving the overall presentation and clarity of the paper.} We also thank Ren-Cang Li for the discussions and helpful suggestions, including a general version of Lemma \ref{lem:inequality}.
  
\def\noopsort#1{}\def\l{\char32l}\def\v#1{{\accent20 #1}}
  \let\^^_=\v\def\hbk{hardback}\def\pbk{paperback}
\providecommand{\href}[2]{#2}
\providecommand{\arxiv}[1]{\href{http://arxiv.org/abs/#1}{arXiv:#1}}
\providecommand{\url}[1]{\texttt{#1}}
\providecommand{\urlprefix}{URL }

\end{document}

%% file: defnAMS.tex
\usepackage{graphicx}
\usepackage{amsmath,amsfonts,amsthm}
\usepackage{amssymb,latexsym}
\usepackage{fixmath}
\usepackage{mathrsfs,amsbsy}
\usepackage{dsfont}
\usepackage{enumerate}

\def\wtd{\widetilde}

\DeclareMathOperator*{\argmin}{argmin}
\DeclareMathOperator*{\argmax}{argmax}

\DeclareMathOperator{\diag}{diag}

\DeclareMathOperator{\HH}{H}
\DeclareMathOperator{\T}{T}

\def\ba{\pmb{a}}
\def\bb{\pmb{b}}
\def\bc{\pmb{c}}

\def\be{\pmb{e}}
\def\Bf{\pmb{f}}
\def\bg{\pmb{g}}

\def\bp{\pmb{p}}
\def\bq{\pmb{q}}

\def\bs{\pmb{s}}
\def\bt{\pmb{t}}
\def\bu{\pmb{u}}
\def\bv{\pmb{v}}
\def\bw{\pmb{w}}
\def\bx{\pmb{x}}
\def\by{\pmb{y}}
\def\bz{\pmb{z}}

\newtheorem{proposition}{Proposition}[section]
\newtheorem{theorem}{Theorem}[section]
\newtheorem{lemma}{Lemma}[section]
\newtheorem{corollary}{Corollary}[section]

\theoremstyle{definition}

\newtheorem{remark}{Remark}[section]
\newtheorem{example}{Example}[section]

\numberwithin{equation}{section}
\numberwithin{figure}{section}
\numberwithin{table}{section}

\allowdisplaybreaks